\documentclass[11pt]{article}

\usepackage[a4paper,margin=26mm]{geometry}
\usepackage{graphicx}
\usepackage{booktabs}
\usepackage{adjustbox}
\usepackage{mathtools}
\usepackage{microtype}
\usepackage{amsthm}
\usepackage{amssymb}
\usepackage[authoryear,round]{natbib}

\usepackage{enumitem}
\usepackage[hidelinks]{hyperref}
\AtBeginDocument{\let\cite\citep}

\numberwithin{equation}{section}

\newtheorem{theorem}{Theorem}
\newtheorem{corollary}{Corollary}
\newtheorem{lemma}{Lemma}
\newtheorem{proposition}{Proposition}

\theoremstyle{definition}
\newtheorem{definition}{Definition}
\newtheorem{assumption}[theorem]{Assumption}

\theoremstyle{remark}
\newtheorem{remark}{Remark}

\title{Prediction Suboptimality of the Lasso in Sparse Linear Regression}
\author{Guo Liu\\
	\small Department of Pure and Applied Mathematics,\\
	\small Waseda University, Japan}
\date{}

\begin{document}
	
	\maketitle
	\begin{abstract}
		The choice of the tuning parameter in the Lasso is central to its statistical performance in high dimensional linear regression. In this work, we study tuning regimes under which the Lasso exhibits suboptimal prediction performance, in the sense that a simple refinement improves upon it both on high probability events and in mean squared prediction error. Our lower bounds identify Gaussian maxima on the selected or localized design as key stochastic quantities in sufficient conditions for prediction improvement, suggesting that the intrinsic Gaussian complexity of the design can provide a more refined perspective on Lasso tuning than the universal rate alone. We further illustrate how structural factors in the design matrix can influence the suboptimality phenomenon and discuss extensions to more general noise structures.
	\end{abstract}
	

	\section{Introduction}
	\label{sec:introduction}
	
	Linear regression is a simple yet powerful framework for high dimensional modeling, valued for its interpretability and transparent representation of signal and noise. The Lasso \cite{tibshirani1996regression} is one of the most widely used methods for estimation and variable selection in high dimensional linear regression. By combining these tasks through an \(\ell_1\) penalty, it yields computationally efficient estimators with well developed theoretical guarantees. A large body of work has established oracle inequalities, prediction error bounds, and sparsity recovery results under various assumptions; see, for example, \cite{geer2016estimation} and the references therein. Within this literature, it is common to choose the tuning parameter at the universal scale \(\sigma\sqrt{(\log p)/n}\) under normalized columns and Gaussian noise. This scale arises from controlling the largest empirical correlation between the design columns and the noise. While such guarantees are fundamental, an upper bound on prediction error does not determine whether the shrinkage induced by a particular tuning choice is necessary.
	
	The choice of the tuning parameter can be more delicate when the design matrix exhibits strong correlations. Existing studies of Lasso prediction under correlated designs suggest that smaller tuning parameters can yield sharper guarantees in certain regimes \cite{vandegeer2013lasso,hebiri2013correlations,dalalyan2017prediction}. Related developments in high dimensional regularization further illustrate the importance of appropriately calibrated penalties for prediction performance \cite{bellec2018slope}. These results indicate that the conventional universal scale may be conservative when the dependence structure of the design is taken into account. However, they primarily concern prediction risk upper bounds. We take a complementary perspective by investigating when, and under what conditions, Lasso tuning induces prediction error that can be reduced by a simple refinement.
	
	To study this prediction suboptimality phenomenon, we compare the Lasso estimator \(\hat{\beta}_{\mathrm{L}}\) with a centered ridge refinement \(\hat{\beta}_{\mathrm{R}}\) and consider the prediction error gap
	\[
	\Delta
	:=\frac{1}{n}\left\{
	\|X(\hat{\beta}_{\mathrm{L}}-\beta_0)\|_2^2
	-\|X(\hat{\beta}_{\mathrm{R}}-\beta_0)\|_2^2
	\right\}.
	\]
	A positive value of \(\Delta\) means that the refinement achieves smaller prediction error for the realized data. The refinement operates on the Lasso equicorrelation set and penalizes deviations from the first stage Lasso coefficients rather than shrinking them toward zero. Its construction and initial analysis appeared in \cite{lassoridge}, whereas the present work focuses on the theoretical characterization of Lasso prediction suboptimality and its dependence on tuning levels and design geometry. Related approaches include the Relaxed Lasso, least squares refitting, Lasso+mLS, conventional Lasso+Ridge, and other refitting methods that account for coefficient signs \cite{meinshausen2007relaxed,belloni2013least,liu2013asymptotic,chetelat2017optimal,chzhen2019lasso}. In the present work, the centered refinement serves as a concrete comparator for investigating the prediction consequences of Lasso tuning, rather than as a new estimator proposed for its own sake.
	
	Our analysis begins with an eventwise lower bound for \(\Delta\), which separates a positive term arising from Lasso shrinkage and the selected Gram matrix from a stochastic term involving
	\(Z(E):=\|X_E^\top\varepsilon/n\|_\infty\),
	where \(E\) denotes the Lasso equicorrelation set. The bound holds for every positive second stage ridge parameter, removing the restriction on the Gram matrix row sums imposed in the earlier analysis \cite{lassoridge}. Together with spectral refinements, it provides explicit sufficient conditions under which the Lasso exhibits suboptimal prediction performance relative to the centered refinement. In particular, prediction improvement can occur when the Lasso penalty is large relative to the noise fluctuations associated with the selected variables, even if the penalty is of the conventional universal order.
	
	The stochastic term \(Z(E)\) highlights the importance of selection and design dependence. Because \(E\) depends on the response, the selected noise correlations cannot be analyzed as if the selected set were independent of the noise. We address this dependence through localization to a deterministic set \(T_0\subseteq[p]\), using the quantity
	\[
	\nu(T_0):=\mathbb E\|X_{T_0}^\top\varepsilon/n\|_\infty.
	\]
	On the event \(\{E\subseteq T_0\}\), the selected maximum is controlled by the localized maximum. Under Gaussian noise, \(\nu(T_0)\) is the expected maximum of a correlated Gaussian family and provides the center for Gaussian concentration. This localization yields high probability and unconditional prediction risk bounds, with an explicit remainder accounting for failures of support containment.
	
	Importantly, the magnitude of \(\nu(T_0)\) depends not only on the number of candidate variables but also on the geometry of their corresponding design columns. When many columns are concentrated around a small number of representative directions, their Gaussian fluctuations overlap, and the resulting maximum can be substantially smaller than the universal bound based on \(\log p\). We investigate this effect through bounds for clustered and other geometrically structured designs. This perspective is conceptually related to broader approaches to model complexity under correlated designs \cite{zhangliu2026descriptive}, although our analysis concerns Gaussian noise maxima rather than descriptive penalties for model selection.
	
	We provide numerical evidence connecting design geometry, tuning levels, and prediction gains. Complementary results address sign preservation of the refinement, prediction improvement under stable signed models, and extensions to martingale difference noise. Together, these results establish a finite sample perspective on Lasso prediction suboptimality, revealing how selection localization and structural dependence can influence the stochastic scale relevant to prediction improvement beyond the conventional universal tuning rate.
	
	\section{Preliminaries and Background}
	\label{sec:preliminaries-background}
	
	\subsection{Notation and Model Assumptions}
	\label{subsec:notation-model-assumptions}
	
	The cardinality of a set \(E\) is denoted by \(|E|\). For a vector \(x\), define  
	\(\|x\|_1 = \sum_i |x_i|\) and \(\|x\|_2 = \bigl(\sum_i x_i^2\bigr)^{1/2}\).
	For a matrix \(A = (A_{ij}) \in \mathbb{R}^{n \times p}\), define
	\begin{align*}
		\|A\|_\infty &= \max_{i} \sum_{j} |A_{ij}|, \\
		\|A\|_2 &= \sup_{\|x\|_2 = 1}\, \|Ax\|_2 .
	\end{align*}
	For any \(S \subseteq [p]\), \(A_S\) denotes the submatrix consisting of the columns of \(A\) indexed by \(S\), \(\beta_S\) denotes the corresponding subvector of \(\beta\), and \(\beta_{-S}\) denotes the subvector of \(\beta\) indexed by the complement \(S^c\).
	We adopt the standard conventions for zero dimensional vectors and matrices. In particular, a \(0\times n\) matrix and a \(0\times 1\) vector are well defined, and the \(0\times 0\) matrix is treated as a valid square matrix whose inverse is itself. Under these conventions, matrix multiplication involving zero dimensional objects is well defined and dimensionally consistent.
	We work with the linear model
	\[
	y = X\beta_0 + \varepsilon,
	\]
	where \(y \in \mathbb{R}^n\), \(X \in \mathbb{R}^{n\times p}\) is a fixed design matrix, and \(\beta_0 \in \mathbb{R}^p\). The noise vector \(\varepsilon \in \mathbb{R}^n\) is arbitrary and no specific distributional assumption is imposed unless stated otherwise. Throughout the probability and expectation analysis, the first stage tuning parameter \(\lambda_{\mathrm L}\) is deterministic.
	
	\subsection{Lasso and Ridge: Definitions and Basic Properties}
	\label{subsec:lasso-ridge-definitions}
	
	\begin{definition}[Lasso estimator]
		Let \(\lambda_{\mathrm{L}} > 0\) be a regularization parameter. 
		The Lasso estimator \(\hat{\beta}_{\mathrm{L}} \in \mathbb{R}^p\) is defined as
		\[
		\hat{\beta}_{\mathrm{L}}
		\in 
		\arg\min_{\beta \in \mathbb{R}^{p}}
		\left\{
		\frac{1}{2n} \| y - X\beta \|_2^2 
		+ \lambda_{\mathrm{L}} \|\beta \|_1
		\right\}.
		\]
	\end{definition}
	
	\begin{definition}
		We define the key quantities associated with the Lasso:
		\begin{itemize}
			\item \emph{The equicorrelation set:}
			\[
			E \coloneqq \{\, i \in \{1,2,\ldots,p\} : |\frac{1}{n}X_i^\top (y - X\hat{\beta}_{\mathrm{L}})| = \lambda_{\mathrm{L}} \,\}.
			\]
			\item \emph{Restricted Gram matrix:}
			\[
			\Sigma_{n,E} \coloneqq \frac{X_{E}^{\top} X_{E}}{n}.
			\]
			\item \emph{The equicorrelation signs:}
			\[
			s \coloneqq \operatorname{sign}({X_E}^{\top}(y - X\hat{\beta}_{\mathrm{L}})).
			\]
		\end{itemize}
	\end{definition}
	
	\begin{proposition}\label{prop:lasso-kkt}
		The Karush--Kuhn--Tucker (KKT) conditions for the Lasso estimator can be written as follows; see, for example, \cite{geer2016estimation}:
		\[
		\frac{1}{n} X^\top (y - X \hat{\beta}_{\mathrm{L}}) = \lambda_{\mathrm{L}} \bar{z},
		\]
		where \( \bar{z} \in \partial \|\hat{\beta}_{\mathrm{L}}\|_1 \), whose \(i\)-th component is given by
		\[
		\bar{z}_i =
		\begin{cases}
			\mathrm{sign}((\hat{\beta}_{\mathrm{L}})_i), & \text{if } (\hat{\beta}_{\mathrm{L}})_i \neq 0, \\[6pt]
			u_i \in [-1, 1], & \text{if } (\hat{\beta}_{\mathrm{L}})_i = 0,
		\end{cases}
		\qquad \text{for } i = 1, \dots, p.
		\]
		Moreover, it follows that
		\[
		\frac{1}{n} X_E^\top (y - X \hat{\beta}_{\mathrm{L}}) = \lambda_{\mathrm{L}} s.
		\]
		Here, \( \mathrm{sign}(\cdot) \) denotes the sign function. In addition, the KKT
		conditions on the active block imply
		\[
		\hat{\beta}_{\mathrm{L},E}
		= (X_E^{\top}X_E)^{-1}\!\left(X_E^{\top}y - n\lambda_{\mathrm{L}} s\right).
		\]
		Here, \((\hat{\beta}_{\mathrm{L}})_i\) denotes the \(i\)-th component of \(\hat{\beta}_{\mathrm{L}}\), and for an index set \(E\), \(\hat{\beta}_{\mathrm{L},E}\) denotes the subvector restricted to \(E\). 
		The explicit expression for \(\hat{\beta}_{\mathrm{L},E}\) holds when \(X_E\) has full column rank (see Proposition~\ref{uniqueness}).
	\end{proposition}
	
	\begin{assumption}[Column normalization]\label{ass:normalization}
		The columns of the design matrix \(X \in \mathbb{R}^{n\times p}\) are normalized as
		\[
		\|X_j\|_2^2 = n, \qquad j=1,\dots,p.
		\]
	\end{assumption}
	
	\begin{assumption}\label{ass:general-position}
		The columns of the design matrix \(X\) are in general position.  
		More precisely, for any integer \( k < \min\{n,p\} \), no \( k \)-dimensional affine subspace of \( \mathbb{R}^n \) contains more than \( k+1 \) elements of the signed set \( \{\pm X_1, \ldots, \pm X_p\} \), excluding antipodal pairs. 
		Although the definition may appear somewhat technical, this condition holds with probability one when the entries of \(X\) are drawn from a continuous distribution.
	\end{assumption}
	
	\begin{proposition}[Uniqueness of the Lasso]\label{uniqueness}
		Under Assumption~\ref{ass:general-position}, for any \(\lambda>0\),
		the Lasso solution \(\hat{\beta}_{\mathrm{L}}\) is unique. Moreover,
		for almost every \(y \in \mathbb{R}^n\), its support is exactly \(E\);
		that is, \(\hat{\beta}_j \neq 0\) if and only if \(j \in E\).
		Furthermore, \(\mathrm{null}(X_E) = \{0\}\); see \cite{tibshirani2013lasso}.
	\end{proposition}
	
	\begin{remark}
		Proposition~\ref{uniqueness} ensures that the definition of the refined estimator is unambiguous. Throughout this paper, we impose Assumptions~\ref{ass:normalization} and~\ref{ass:general-position} without further mention, except that the bounds on Gaussian complexity do not require general position.
	\end{remark}
	
	\subsection{The Lasso--Ridge Estimator and Related Definitions}
	\label{subsec:lasso-ridge-estimator}
	
	\begin{definition}
		A Lasso--Ridge estimator with \(\lambda_{\mathrm{R}} \geq 0\) is defined as
		\begin{align*}
			\hat{\beta}_{\mathrm{R}} =\underset{\beta \in \mathbb{R}^{p}, \beta_{-E}=0}{\arg\min} \left\{ \frac{1}{2n} \| y - X\beta \|_2^2 +  \frac{\lambda_{\mathrm{R}}}{2}\|\beta-\hat{\beta}_{\mathrm{L}} \|_2^2  \right\},
		\end{align*}
		or equivalently:
		\begin{align*}
			\hat{\beta}_{\mathrm{R}} = \hat{\delta} + \hat{\beta}_{\mathrm{L}},
		\end{align*}
		\begin{align*}
			\quad \hat{\delta} = \underset{\delta \in \mathbb{R}^ {p}, \delta_{-E}=0}{\arg\min} \left\{ \frac{1}{2n} \| y -  X\hat{\beta}_{\mathrm{L}}-X\delta \|_2^2 +  \frac{\lambda_{\mathrm{R}}}{2}\| \delta \|_2^2 \right\}.
		\end{align*}
	\end{definition}
	
	\begin{definition}
		Let \(\hat{\beta}_{\mathrm{L}}\) and \(\hat{\beta}_{\mathrm{R}}\) denote the Lasso and Lasso--Ridge estimators of \(\beta_0\), respectively.
		We define the prediction error gap:
		\[
		\Delta(\hat{\beta}_{\mathrm{L}},\hat{\beta}_{\mathrm{R}})
		:= \frac{1}{n}(
		\|X(\hat{\beta}_{\mathrm{L}}-\beta_0)\|_2^2
		- \|X(\hat{\beta}_{\mathrm{R}}-\beta_0)\|_2^2
		).
		\]
		Define the difference in prediction mean squared error:
		\[
		\mathrm{DMSE}
		:= \frac{1}{n}\,\mathbb{E}[
		\|X(\hat{\beta}_{\mathrm{L}}-\beta_0)\|_2^2
		- \|X(\hat{\beta}_{\mathrm{R}}-\beta_0)\|_2^2
		].
		\]
		A positive value of \(\Delta(\hat{\beta}_{\mathrm{L}},\hat{\beta}_{\mathrm{R}})\) means that \(\hat{\beta}_{\mathrm{R}}\) has smaller prediction error than \(\hat{\beta}_{\mathrm{L}}\) pointwise in the data, while a positive value of \(\mathrm{DMSE}\) means that \(\hat{\beta}_{\mathrm{R}}\) has smaller prediction mean squared error in expectation.
	\end{definition}
	
	\begin{definition}[Projection operators]
		Let \(A \in \mathbb{R}^{n \times k}\) be a matrix with full column rank, i.e., \(\operatorname{rank}(A)=k\). The orthogonal projection onto the column space of \(A\) is defined by
		\begin{align*}
			P_{A} := A(A^{\top}A)^{-1}A^{\top}.
		\end{align*}
		The projection onto the orthogonal complement of \(\operatorname{col}(A)\) is
		\begin{align*}
			P_{A}^{\perp} := I_n - P_{A}.
		\end{align*}
		For any \(u \in \mathbb{R}^n\), \(P_{A}u\) is the projection of \(u\) onto \(\operatorname{col}(A)\), and \(P_{A}^{\perp}u\) is the corresponding orthogonal residual. In particular, if \(X_E \in \mathbb{R}^{n \times |E|}\) denotes the submatrix of \(X\) with columns indexed by \(E\), and \(\operatorname{rank}(X_E)=|E|\), then
		\begin{align*}
			P_{X_E} &= X_E(X_E^{\top}X_E)^{-1}X_E^{\top}, \\
			P_{X_E}^{\perp} &= I_n - P_{X_E}.
		\end{align*}
	\end{definition}
	
	\section{Main Results}
	\label{sec:main-results}
	
	\subsection{Eventwise mechanism and high probability consequences}
	\label{subsec:high-probability-lower-bounds}
	
	\begin{theorem}[Eventwise lower bound]\label{thm:eventwise-delta-lower}
		Assume \(\lambda_{\mathrm R}>0\). We use the empty set convention
		\[
		\hat\delta=0,\qquad 
		\Delta(\hat\beta_{\mathrm L},\hat\beta_{\mathrm R})=0,
		\qquad
		\|\Sigma_{n,E}\|_2=\|\Sigma_{n,E}\|_\infty=0
		\quad\text{when }E=\emptyset .
		\]
		Then
		\begin{align*}
			\Delta(\hat\beta_{\mathrm L},\hat\beta_{\mathrm R})
			\geq
			\lambda_{\mathrm L}^2 |E|
			\frac{\|\Sigma_{n,E}\|_2+2\lambda_{\mathrm R}}
			{(\|\Sigma_{n,E}\|_2+\lambda_{\mathrm R})^2}
			-
			\frac{2\lambda_{\mathrm L}|E|}{n\lambda_{\mathrm R}}
			\bigl\|X_E^\top\varepsilon\bigr\|_\infty .
		\end{align*}
		Consequently, on the event
		\[
		\mathcal E_\tau(E):=
		\{
		\|X_E^\top\varepsilon/n\|_\infty\leq \tau\lambda_{\mathrm L}
		\},
		\]
		where \(\tau>0\), we have
		\begin{align*}
			\Delta(\hat\beta_{\mathrm L},\hat\beta_{\mathrm R})
			\geq
			\lambda_{\mathrm L}^2 |E|
			\left(
			\frac{\|\Sigma_{n,E}\|_2+2\lambda_{\mathrm R}}
			{(\|\Sigma_{n,E}\|_2+\lambda_{\mathrm R})^2}
			-
			\frac{2\tau}{\lambda_{\mathrm R}}
			\right).
		\end{align*}
	\end{theorem}
	
	\begin{proof}
		If \(E=\emptyset\), then \(\hat\delta=0\), and hence
		\[
		\Delta(\hat\beta_{\mathrm L},\hat\beta_{\mathrm R})=0.
		\]
		Under the empty set conventions, the right hand side of the claimed inequality is also zero because \(|E|=0\). Therefore the result is trivial on \(\{E=\emptyset\}\). We henceforth assume \(E\neq\emptyset\).
		
		We start by rewriting the prediction gap as follows:
		\begin{align}
			\Delta(\hat\beta_{\mathrm L},\hat\beta_{\mathrm R})
			&=\frac1n\|X\beta_0-X\hat\beta_{\mathrm L}\|_2^2
			-\frac1n\|X\beta_0-X\hat\beta_{\mathrm L}-X\hat\delta\|_2^2\notag\\
			&=\frac2n\langle X\beta_0-X\hat\beta_{\mathrm L},X\hat\delta\rangle
			-\frac1n\langle X\hat\delta,X\hat\delta\rangle\notag\\
			&=\frac2n\langle X\beta_0-X_E\hat\beta_{\mathrm L,E},X_E\hat\delta_E\rangle
			-\frac1n\langle X_E\hat\delta_E,X_E\hat\delta_E\rangle.
			\label{eq:1}
		\end{align}
		Here we used the fact that the support of \(\hat\beta_{\mathrm L}\) is contained in the equicorrelation set \(E\), while \(\hat\delta_{-E}=0\). By Proposition~\ref{uniqueness}, \(X_E\) has full column rank. Recalling the Lasso KKT conditions in Proposition~\ref{prop:lasso-kkt}, we have
		\[
		\hat\beta_{\mathrm L,E}
		=(X_E^\top X_E)^{-1}(X_E^\top y-n\lambda_{\mathrm L}s).
		\]
		The first term in the inner product in \eqref{eq:1} can be rewritten using \(y=X\beta_0+\varepsilon\):
		\begin{align*}
			X\beta_0-X_E\hat\beta_{\mathrm L,E}
			&=X\beta_0-X_E(X_E^\top X_E)^{-1}
			(X_E^\top y-n\lambda_{\mathrm L}s)\\
			&=X\beta_0-X_E(X_E^\top X_E)^{-1}
			\bigl(X_E^\top(X\beta_0+\varepsilon)-n\lambda_{\mathrm L}s\bigr)\\
			&=P_{X_E}^{\perp}X\beta_0-P_{X_E}\varepsilon
			+n\lambda_{\mathrm L}X_E(X_E^\top X_E)^{-1}s.
		\end{align*}
		Here \(P_{X_E}=X_E(X_E^\top X_E)^{-1}X_E^\top\) is the orthogonal projection onto \(\operatorname{col}(X_E)\), and \(P_{X_E}^{\perp}=I_n-P_{X_E}\). Substituting this decomposition into \eqref{eq:1} gives
		\begin{align}
			\Delta(\hat\beta_{\mathrm L},\hat\beta_{\mathrm R})
			&=\frac2n\langle P_{X_E}^{\perp}X\beta_0,X_E\hat\delta_E\rangle
			-\frac2n\langle P_{X_E}\varepsilon,X_E\hat\delta_E\rangle\notag\\
			&\quad+\frac2n\langle n\lambda_{\mathrm L}X_E(X_E^\top X_E)^{-1}s,
			X_E\hat\delta_E\rangle
			-\frac1n\langle X_E\hat\delta_E,X_E\hat\delta_E\rangle.
			\label{eq:2}
		\end{align}
		We now simplify the first three inner products in \eqref{eq:2}. For the first, direct matrix multiplication gives
		\begin{align*}
			\langle P_{X_E}^{\perp}X\beta_0,X_E\hat\delta_E\rangle
			&=\hat\delta_E^\top X_E^\top
			\bigl(I_n-X_E(X_E^\top X_E)^{-1}X_E^\top\bigr)X\beta_0\\
			&=\hat\delta_E^\top
			\bigl(X_E^\top-X_E^\top X_E(X_E^\top X_E)^{-1}X_E^\top\bigr)X\beta_0\\
			&=0.
		\end{align*}
		Similarly, the projection identity yields
		\begin{align*}
			\langle P_{X_E}\varepsilon,X_E\hat\delta_E\rangle
			&=\hat\delta_E^\top X_E^\top
			X_E(X_E^\top X_E)^{-1}X_E^\top\varepsilon\\
			&=\langle X_E\hat\delta_E,\varepsilon\rangle.
		\end{align*}
		For the third term, we obtain
		\begin{align*}
			\frac2n\langle n\lambda_{\mathrm L}X_E(X_E^\top X_E)^{-1}s,
			X_E\hat\delta_E\rangle
			&=2\lambda_{\mathrm L}\hat\delta_E^\top
			X_E^\top X_E(X_E^\top X_E)^{-1}s\\
			&=2\lambda_{\mathrm L}\langle\hat\delta_E,s\rangle.
		\end{align*}
		Substituting these equalities into \eqref{eq:2}, we obtain
		\begin{equation}
			\Delta(\hat\beta_{\mathrm L},\hat\beta_{\mathrm R})
			=2\lambda_{\mathrm L}\langle\hat\delta_E,s\rangle
			-\frac1n\langle X_E\hat\delta_E,X_E\hat\delta_E\rangle
			-\frac2n\langle X_E\hat\delta_E,\varepsilon\rangle.
			\label{eq:3}
		\end{equation}
		
		Next we derive the closed form expression for the correction. By its definition, when \(E\neq\emptyset\),
		\[
		\hat\delta_E=\underset{u\in\mathbb R^{|E|}}{\arg\min}
		\left\{\frac1{2n}\|y-X\hat\beta_{\mathrm L}-X_Eu\|_2^2
		+\frac{\lambda_{\mathrm R}}2\|u\|_2^2\right\}.
		\]
		The first order condition of this strictly convex problem is
		\[
		-\frac1nX_E^\top\bigl(y-X\hat\beta_{\mathrm L}-X_E\hat\delta_E\bigr)
		+\lambda_{\mathrm R}\hat\delta_E=0.
		\]
		The Lasso KKT conditions give
		\[
		\frac1nX_E^\top(y-X\hat\beta_{\mathrm L})=\lambda_{\mathrm L}s.
		\]
		Since \(\Sigma_{n,E}=X_E^\top X_E/n\), combining the two equations yields
		\[
		(\Sigma_{n,E}+\lambda_{\mathrm R}I_{|E|})\hat\delta_E=\lambda_{\mathrm L}s.
		\]
		As \(\lambda_{\mathrm R}>0\), the matrix is invertible. Thus
		\begin{equation}
			\hat\delta_E
			=\lambda_{\mathrm L}(\Sigma_{n,E}+\lambda_{\mathrm R}I_{|E|})^{-1}s
			=\frac{\lambda_{\mathrm L}}{\lambda_{\mathrm R}}
			\left(\frac{\Sigma_{n,E}}{\lambda_{\mathrm R}}+I_{|E|}\right)^{-1}s.
			\label{eq:ridge-correction-closed-form}
		\end{equation}
		
		The restricted Gram matrix admits the orthogonal decomposition
		\[
		\Sigma_{n,E}=P\operatorname{diag}(\mu_1,\ldots,\mu_{|E|})P^\top,
		\]
		where \(\mu_1\geq\cdots\geq\mu_{|E|}>0\) are its eigenvalues and \(P\) is orthogonal. Put \(M_E:=\mu_1=\|\Sigma_{n,E}\|_2\). By \eqref{eq:ridge-correction-closed-form},
		\begin{align*}
			2\lambda_{\mathrm L}\langle\hat\delta_E,s\rangle
			&=2\lambda_{\mathrm L}^2s^\top P
			\operatorname{diag}\!\left(
			\frac1{\mu_1+\lambda_{\mathrm R}},\ldots,
			\frac1{\mu_{|E|}+\lambda_{\mathrm R}}\right)P^\top s.
		\end{align*}
		Similarly,
		\begin{align*}
			\frac1n\langle X_E\hat\delta_E,X_E\hat\delta_E\rangle
			&=\lambda_{\mathrm L}^2s^\top
			(\Sigma_{n,E}+\lambda_{\mathrm R}I_{|E|})^{-1}
			\Sigma_{n,E}
			(\Sigma_{n,E}+\lambda_{\mathrm R}I_{|E|})^{-1}s\\
			&=\lambda_{\mathrm L}^2s^\top P
			\operatorname{diag}\!\left(
			\frac{\mu_1}{(\mu_1+\lambda_{\mathrm R})^2},\ldots,
			\frac{\mu_{|E|}}{(\mu_{|E|}+\lambda_{\mathrm R})^2}\right)P^\top s.
		\end{align*}
		For notational convenience, define
		\begin{align*}
			H_{E,s}(\lambda_{\mathrm L},\lambda_{\mathrm R})
			&:=2\lambda_{\mathrm L}\langle\hat\delta_E,s\rangle
			-\frac1n\langle X_E\hat\delta_E,X_E\hat\delta_E\rangle\\
			&=\lambda_{\mathrm L}^2s^\top P
			\operatorname{diag}\!\left(
			\frac{\mu_1+2\lambda_{\mathrm R}}{(\mu_1+\lambda_{\mathrm R})^2},\ldots,
			\frac{\mu_{|E|}+2\lambda_{\mathrm R}}{(\mu_{|E|}+\lambda_{\mathrm R})^2}
			\right)P^\top s.
		\end{align*}
		The function \(x\mapsto(x+2\lambda_{\mathrm R})/(x+\lambda_{\mathrm R})^2\) is decreasing on \([0,\infty)\). Since \(\mu_i\leq M_E\) and \(\|P^\top s\|_2^2=\|s\|_2^2=|E|\), we obtain
		\begin{equation*}
			H_{E,s}(\lambda_{\mathrm L},\lambda_{\mathrm R})
			\geq\lambda_{\mathrm L}^2|E|
			\frac{M_E+2\lambda_{\mathrm R}}{(M_E+\lambda_{\mathrm R})^2}.
		\end{equation*}
		
		For the stochastic term, Cauchy--Schwarz and the positive semidefiniteness of \(\Sigma_{n,E}\) imply
		\begin{align*}
			\|\hat\delta_E\|_1
			&\leq\sqrt{|E|}\,\|\hat\delta_E\|_2\\
			&\leq\lambda_{\mathrm L}|E|
			\|(\Sigma_{n,E}+\lambda_{\mathrm R}I_{|E|})^{-1}\|_2\\
			&\leq\frac{\lambda_{\mathrm L}|E|}{\lambda_{\mathrm R}}.
		\end{align*}
		Writing \(Z(E):=\|X_E^\top\varepsilon/n\|_\infty\), we therefore have
		\begin{equation}
			\frac2n\langle X_E\hat\delta_E,\varepsilon\rangle
			=\frac2n\langle\hat\delta_E,X_E^\top\varepsilon\rangle
			\leq2\|\hat\delta_E\|_1Z(E)
			\leq\frac{2\lambda_{\mathrm L}|E|}{\lambda_{\mathrm R}}Z(E).
			\label{eq:5}
		\end{equation}
		Combining \eqref{eq:3}--\eqref{eq:5} proves the first inequality. The second follows on \(\mathcal E_\tau(E)\) by substituting \(Z(E)\leq\tau\lambda_{\mathrm L}\).
	\end{proof}
	
	The eventwise inequality separates the contribution of Lasso regularization, measured through the selected Gram matrix, from the stochastic term \(Z(E)\), the maximum empirical noise correlation on the equicorrelation set. Controlling this term yields sufficient conditions for a positive prediction gain without first introducing a global noise bound.
	
	Because the second step is performed after Lasso selection, the set \(E\) and the restricted Gram matrix \(\Sigma_{n,E}\) are available at refitting time. Hence \(\lambda_{\mathrm R}\) may depend on the selected design, and the theorem requires only that it be positive.
	
	Theorem~\ref{thm:eventwise-delta-lower} already contains the main stochastic message: prediction improvement is favored when the Lasso penalty is large relative to the selected support noise maximum. The next result recovers the usual global high probability formulation by replacing the random selected support maximum with its full design counterpart.
	
	\begin{corollary}[Global high probability lower bound]\label{thm:hp-delta-lower}
		Assume that \(\lambda_{\mathrm R}>0\), and let
		\[
		\mathcal E_\tau :=
		\{
		\|X^\top\varepsilon/n\|_\infty \leq \tau\lambda_{\mathrm L}
		\},
		\]
		where \(\tau>0\). Then, on the event \(\mathcal E_\tau\),
		\begin{align*}
			\Delta(\hat\beta_{\mathrm L},\hat\beta_{\mathrm R})
			\geq
			\lambda_{\mathrm L}^2 |E|
			\left(
			\frac{\|\Sigma_{n,E}\|_2+2\lambda_{\mathrm R}}{(\|\Sigma_{n,E}\|_2+\lambda_{\mathrm R})^2}
			-
			\frac{2\tau}{\lambda_{\mathrm R}}
			\right).
		\end{align*}
	\end{corollary}
	
	\begin{proof}
		On the event \(\mathcal E_\tau\), we have
		\[
		\|X_E^\top\varepsilon/n\|_\infty
		\leq
		\|X^\top\varepsilon/n\|_\infty
		\leq
		\tau\lambda_{\mathrm L}.
		\]
		Hence \(\mathcal E_\tau \subseteq \mathcal E_\tau(E)\), and the claim follows from Theorem~\ref{thm:eventwise-delta-lower}.
	\end{proof}
	
	The global event is convenient but may be conservative when selected variables lie in a smaller deterministic subset of the design. This motivates the localized formulation.
	
	\begin{corollary}[Localized high probability lower bound under support containment]
		\label{cor:hp-delta-lower-T0}
		Assume that \(\lambda_{\mathrm R}>0\), and let \(T_0\subseteq\{1,\dots,p\}\) be deterministic. Define
		\[
		\mathcal E_\tau(T_0) :=
		\{
		\|X_{T_0}^\top\varepsilon/n\|_\infty \leq \tau\lambda_{\mathrm L}
		\},
		\]
		where \(\tau>0\). Then, on the event
		\[
		\mathcal E_\tau(T_0)\cap\{E\subseteq T_0\},
		\]
		we have
		\begin{align*}
			\Delta(\hat\beta_{\mathrm L},\hat\beta_{\mathrm R})
			\geq
			\lambda_{\mathrm L}^2 |E|
			\left(
			\frac{\|\Sigma_{n,E}\|_2+2\lambda_{\mathrm R}}{(\|\Sigma_{n,E}\|_2+\lambda_{\mathrm R})^2}
			-
			\frac{2\tau}{\lambda_{\mathrm R}}
			\right).
		\end{align*}
	\end{corollary}
	
	\begin{proof}
		On the event \(\{E\subseteq T_0\}\), we have
		\[
		\|X_E^\top\varepsilon/n\|_\infty
		\leq
		\|X_{T_0}^\top\varepsilon/n\|_\infty.
		\]
		Hence, on \(\mathcal E_\tau(T_0)\cap\{E\subseteq T_0\}\),
		\[
		\|X_E^\top\varepsilon/n\|_\infty
		\leq
		\tau\lambda_{\mathrm L}.
		\]
		Thus
		\[
		\mathcal E_\tau(T_0)\cap\{E\subseteq T_0\}\subseteq \mathcal E_\tau(E),
		\]
		and the claim follows from Theorem~\ref{thm:eventwise-delta-lower}.
	\end{proof}
	
	\begin{remark}[Localized noise level and stochastic complexity]\label{rem:localized-stochastic-complexity}
		For a deterministic set \(T\subseteq\{1,\dots,p\}\), define
		\[
		\nu(T)
		:=
		\mathbb E\bigl[\|X_T^\top\varepsilon/n\|_\infty\bigr].
		\]
		Corollary~\ref{cor:hp-delta-lower-T0} shows why this quantity is natural. On \(\{E\subseteq T_0\}\), the selected support noise maximum is bounded by the localized maximum on \(T_0\). Hence \(\nu(T_0)\) summarizes the localized noise level associated with the deterministic class used to contain selection.
		
		At the most elementary level, Markov's inequality gives
		\[
		\mathbb P\!\left(
		\|X_{T_0}^\top\varepsilon/n\|_\infty>\tau\lambda_{\mathrm L}
		\right)
		\leq
		\frac{\nu(T_0)}{\tau\lambda_{\mathrm L}}.
		\]
		Under Gaussian noise, the connection is stronger. Writing \(\varepsilon=\sigma g\), the map
		\(g\mapsto\|X_{T_0}^\top(\sigma g)/n\|_\infty\) is \(\sigma/\sqrt n\)-Lipschitz under column normalization. Gaussian concentration therefore yields
		\[
		\mathbb P\!\left(
		\|X_{T_0}^\top\varepsilon/n\|_\infty
		>
		\nu(T_0)+t
		\right)
		\leq
		\exp\!\left(-\frac{nt^2}{2\sigma^2}\right),
		\qquad t>0.
		\]
		Thus \(\nu(T_0)\) is not only an expectation quantity; under Gaussian noise it is also the center of the localized high probability fluctuation.
		
		The containment requirement itself is weaker than exact variable selection consistency. Under Gaussian noise, Proposition~\ref{uniqueness} identifies \(E\) with the Lasso support almost surely. Thus, for \(S_0:=\{j:\beta_{0,j}\neq0\}\), the event \(E\subseteq S_0\) corresponds almost surely to absence of false inclusions; see \citet[Theorem~11.3(b), p.~302]{hastie2015statistical} for sufficient conditions. One may take \(T_0=S_0\) in this case, or a deterministic superset of \(S_0\) when only coarse localization is available. The same localization principle will reappear in the expectation result below, where failure of containment is penalized explicitly rather than excluded eventwise.
	\end{remark}
	
	\paragraph{Magnitude of the lower bound and selected support geometry.}
	The stochastic noise maximum determines whether the guaranteed improvement is positive. Separately, the selected Gram matrix controls the size of the bias correction. Theorem~\ref{thm:eventwise-delta-lower} holds for every \(\lambda_{\mathrm R}>0\), so scaling the second stage penalty by the largest selected Gram eigenvalue yields a direct spectral guarantee.
	
	\begin{corollary}[Spectral norm adaptive refitting]\label{cor:hp-delta-lower-spectral}
		Suppose that \(E\neq\emptyset\) and set
		\[
		M_E:=\|\Sigma_{n,E}\|_2,\qquad \lambda_{\mathrm R}=cM_E,
		\qquad c>0.
		\]
		For \(\tau>0\), on the event \(\mathcal E_\tau(E)\),
		\[
		\Delta(\hat\beta_{\mathrm L},\hat\beta_{\mathrm R})
		\geq
		\frac{\lambda_{\mathrm L}^2|E|}{M_E}
		\left\{\frac{1+2c}{(1+c)^2}-\frac{2\tau}{c}\right\}.
		\]
	\end{corollary}
	\begin{proof}
		By the choice of \(\lambda_{\mathrm R}\),
		\[
		\frac{M_E+2\lambda_{\mathrm R}}{(M_E+\lambda_{\mathrm R})^2}
		=\frac{1+2c}{M_E(1+c)^2},\qquad
		\frac{2\tau}{\lambda_{\mathrm R}}=\frac{2\tau}{cM_E}.
		\]
		Apply Theorem~\ref{thm:eventwise-delta-lower}. Since the columns are normalized, \(M_E\geq1\) for nonempty \(E\), so the tuning choice is positive.
	\end{proof}
	
	\begin{remark}[Optimizing the spectral lower bound]\label{rem:optimal-spectral-c}
		For \(0<\tau<1\), the lower bound coefficient on \(\mathcal E_\tau(E)\) is maximized over \(c>0\) at
		\[
		c^\star=\frac{\tau^{1/3}}{1-\tau^{1/3}},
		\]
		and the resulting statement is
		\[
		\Delta(\hat\beta_{\mathrm L},\hat\beta_{\mathrm R})
		\geq\frac{\lambda_{\mathrm L}^2|E|}{M_E}
		(1-\tau^{1/3})^2(1+2\tau^{1/3}).
		\]
		Indeed, differentiation of \(b_\tau(c)=(1+2c)/(1+c)^2-2\tau/c\) gives
		\(b_\tau'(c)=-2c/(1+c)^3+2\tau/c^2\), whose unique zero satisfies
		\((c/(1+c))^3=\tau\). This maximizes the \emph{certified lower bound} on the specified event, not necessarily the actual prediction gain. The optimizing \(c\) presumes that \(\tau\) has been fixed in advance. In particular, a strictly positive guarantee is possible for every \(\tau<1\).
		
		Under column normalization, \(M_E\leq\operatorname{tr}(\Sigma_{n,E})=|E|\). Consequently, whenever \(b_\tau(c)\geq0\), the spectral lower bound on \(\mathcal E_\tau(E)\) is at least \(\lambda_{\mathrm L}^2b_\tau(c)\).
	\end{remark}
	
	\begin{corollary}[Prediction gain at the classical sparse rate]\label{cor:two-sided-sparse-gain}
		Let \(S_0=\{j:\beta_{0,j}\neq0\}\), with \(s^\ast=|S_0|\geq1\), and suppose that the standard Lasso oracle prediction inequality
		\[
		\frac1n\|X(\hat\beta_{\mathrm L}-\beta_0)\|_2^2\leq C_{\mathrm{RE}}\lambda_{\mathrm L}^2s^\ast
		\]
		holds on the noise event \(\mathcal E_{1/2}=\{\|X^\top\varepsilon/n\|_\infty\leq\lambda_{\mathrm L}/2\}\), as is standard under a suitable restricted eigenvalue condition; see, for example, \citet[Section~11.3]{hastie2015statistical}. Suppose also that, on the event under consideration,
		\[
		|E|\geq\kappa s^\ast,\qquad \|\Sigma_{n,E}\|_2\leq M,
		\]
		for fixed \(\kappa,M>0\). Choose \(\lambda_{\mathrm R}=4\|\Sigma_{n,E}\|_2\) when \(E\neq\emptyset\). Then, on the joint event \(\mathcal E_{1/2}\cap\{|E|\geq\kappa s^\ast,\ \|\Sigma_{n,E}\|_2\leq M\}\),
		\[
		\frac{11\kappa}{100M}\lambda_{\mathrm L}^2s^\ast\leq\Delta(\hat\beta_{\mathrm L},\hat\beta_{\mathrm R})\leq C_{\mathrm{RE}}\lambda_{\mathrm L}^2s^\ast.
		\]
	\end{corollary}
	\begin{proof}
		Since \(\mathcal E_{1/2}\subseteq\mathcal E_{1/2}(E)\), we may apply Corollary~\ref{cor:hp-delta-lower-spectral} with \(c=4\) and \(\tau=1/2\). Then
		\[
		\frac{1+2c}{(1+c)^2}-\frac{2\tau}{c}=\frac9{25}-\frac14=\frac{11}{100}.
		\]
		The conditions on the selected model size and the largest eigenvalue of the selected Gram matrix yield the lower bound. The upper bound follows from nonnegativity of the refined prediction loss and the assumed Lasso oracle inequality:
		\[
		\Delta\leq\frac1n\|X(\hat\beta_{\mathrm L}-\beta_0)\|_2^2\leq C_{\mathrm{RE}}\lambda_{\mathrm L}^2s^\ast.
		\]
	\end{proof}
	\begin{remark}[Sparse prediction rates]\label{rem:minimax-scale-comparison}
		Under universal tuning \(\lambda_{\mathrm L}\asymp\sigma\sqrt{\log p/n}\), Corollary~\ref{cor:two-sided-sparse-gain} yields a prediction gain of order \(\sigma^2s^\ast\log p/n\) on the specified joint event, which is the same order as the classical Lasso oracle prediction bound. Moreover, on this event, the gain is at least a fraction \(11\kappa/(100MC_{\mathrm{RE}})\) of the Lasso prediction loss. Thus, with fixed constants, a nonvanishing fraction of the original loss can be removed even when the loss itself vanishes asymptotically.
		
		The usual sparse minimax prediction benchmark under suitable design regularity conditions is \(\sigma^2s^\ast\log(ep/s^\ast)/n\); see, for example, \citet{raskutti2011minimax,bellec2018slope}. Since \(\log(ep/s^\ast)\leq 1+\log p\), the prediction gain on the specified joint event is not of smaller order than this benchmark. In particular, these orders agree when \(\log p\asymp\log(ep/s^\ast)\). The resulting prediction gain is compatible with minimax rate optimality, since the minimax benchmark concerns worst case expected prediction risk over a sparse parameter class rather than the pointwise loss of a particular estimator or the gap between two procedures. The comparison does not by itself establish an unconditional expected gain of the same order.
		
		The selection size and spectral conditions ensure that \(|E|/\|\Sigma_{n,E}\|_2\) is at least of order \(s^\ast\), which is the relevant factor in the spectral lower bound. The size condition holds, for example, when selection retains a fixed fraction of the true support, while the spectral condition holds when the selected model size lies within a range over which upper sparse eigenvalues are uniformly bounded. These conditions supplement the restricted eigenvalue assumption underlying the usual oracle inequality and are not implied by it alone.
	\end{remark}

	\subsection{Expectation lower bounds: the same localized noise level in mean}
	\label{subsec:expectation-lower-bounds}
	
	\begin{theorem}[Expectation lower bound]\label{thm:lasso-dmse-lower-bound}
		Choose \(\lambda_{\mathrm R}=c|E|\) on \(E\neq\emptyset\), where \(c>0\), and set \(\hat\delta=0\) otherwise. Suppose
		\(\mathbb E\|X^\top\varepsilon/n\|_\infty^2<\infty\). Then
		\begin{align*}
			\mathrm{DMSE}\geq\sup_{T_0\subseteq[p]}\frac{2\lambda_{\mathrm L}}{c}
			\Biggl\{&\frac{c(2c+1)}{2(c+1)^2}\lambda_{\mathrm L}
			\mathbb P(E\neq\emptyset)-\nu(T_0)\\
			&-\left(\mathbb E\|X^\top\varepsilon/n\|_\infty^2\right)^{1/2}
			\sqrt{\mathbb P(E\not\subseteq T_0)}\Biggr\}.
		\end{align*}
		No independence between \(E\) and \(\varepsilon\) is assumed.
	\end{theorem}
	
	\begin{remark}[Interpretation of the expectation bound]\label{rem:expectation-localization}
		The expectation bound follows from the same eventwise mechanism. A deterministic set \(T_0\) localizes the noise maximum, while Cauchy--Schwarz controls the contribution from failures of support containment. The second stage choice \(\lambda_{\mathrm R}=c|E|\) provides a simple bound under the standing design assumptions; it is not proposed as an optimal practical tuning rule.
		The quantity \(\nu(T_0)=\mathbb E\|X_{T_0}^\top\varepsilon/n\|_\infty\) is the same localized noise level as in Corollary~\ref{cor:hp-delta-lower-T0}. If \(\mathbb P(E\not\subseteq T_0)\leq\eta\), the support containment penalty is at most \(\bigl(\mathbb E\|X^\top\varepsilon/n\|_\infty^2\bigr)^{1/2}\sqrt\eta\). Taking \(T_0=[p]\) removes this penalty altogether. This simple form is generally less sharp than a Gaussian tail refinement on a rare event, but requires only a finite second moment.
	\end{remark}
	
	\begin{remark}[Gaussian moment scale]
		If \(\varepsilon\sim N(0,\sigma^2 I_n)\), column normalization implies
		\[
		\mathbb E\|X^\top\varepsilon/n\|_\infty^2
		\leq\frac{\sigma^2\{2\log(4p)+1\}}{n};
		\]
		see Lemma~\ref{lem:gaussian-infty-sq-bound} in Appendix~\ref{appA}.
		For a deterministic \(T_0\neq\emptyset\), the usual Gaussian maximal inequality gives
		\(\nu(T_0)\leq\sigma\sqrt{2\log(2|T_0|)/n}\); Section~\ref{subsec:concrete-gaussian-complexity} gives smaller bounds under additional design structure. The probability of an empty equicorrelation set is discussed in Lemma~\ref{lem:eventE-empty-prob}.
	\end{remark}
	
	\begin{corollary}[Global expectation lower bound]
		Choose \(\lambda_{\mathrm R}=c|E|\) on \(E\neq\emptyset\), where \(c>0\), and use the zero correction otherwise. If the relevant expectations are finite, then
		\[
		\mathrm{DMSE}\geq
		\frac{2\lambda_{\mathrm L}}{c}\left\{
		\frac{c(2c+1)}{2(c+1)^2}\lambda_{\mathrm L}
		\mathbb P(E\neq\emptyset)-\nu([p])
		\right\}.
		\]
	\end{corollary}
	
	\begin{remark}[Global tuning comparison]
		Put \(g(c)=c(2c+1)/(2(c+1)^2)\). As \(c\to\infty\), \(g(c)\uparrow1\). Thus, if \(\mathbb P(E\neq\emptyset)>0\), the global corollary certifies a positive expected gain for some \(c>0\) whenever
		\[
		\lambda_{\mathrm L}>
		\frac{\nu([p])}{\mathbb P(E\neq\emptyset)}.
		\]
		This is a sufficient condition on the bound, not a claim that taking \(c\) very large optimizes the actual prediction gain. The expectation result is intended as a complementary tuning comparison.
	\end{remark}
	
	\subsection{Localized stochastic complexity and departures from the universal scale}
	\label{subsec:concrete-gaussian-complexity}
	
	The localized noise level \(\nu(T_0)\) can differ from the universal scale through two mechanisms: restriction to the index set \(T_0\) and overlap among the Gaussian coordinates indexed by that set. The first is related to support containment in Lasso selection theory, as discussed in Remark~\ref{rem:localized-stochastic-complexity}. The propositions below quantify the second mechanism in several geometric settings.
	
	To make the second mechanism explicit, throughout this subsection suppose that \(\varepsilon\sim\mathcal N(0,\sigma^2I_n)\), \(\|X_j\|_2^2=n\), and define
	\[
	\mathcal G(X):=\frac1n\mathbb E\|X^\top\varepsilon\|_\infty.
	\]
	For a deterministic \(T_0\), \(\mathcal G(X_{T_0})=\nu(T_0)\). Each proposition therefore applies directly to the localized design. Proofs appear in Appendix~\ref{app:stochastic-term}.
	
	\begin{proposition}[Clustered columns]\label{prop:clustered-gaussian-max}
		Suppose that \(\{1,\dots,p\}\) is partitioned into \(K\) clusters. For each cluster choose a representative \(v_\ell\in\mathbb R^n\) satisfying \(\|v_\ell\|_2\leq\sqrt n\), and let \(r(j)\) denote the cluster containing column \(j\). If
		\[
		\|X_j-v_{r(j)}\|_2\leq\delta\sqrt n,
		\qquad j=1,\dots,p,
		\]
		for some \(\delta\geq0\), then
		\[
		\mathcal G(X)
		\leq
		\frac{\sigma}{\sqrt n}
		\left\{
		\sqrt{2\log(2K)}
		+
		\delta\sqrt{2\log(2p)}
		\right\}.
		\]
	\end{proposition}
	
	\begin{proposition}[Approximate low dimensional design]\label{prop:approx-low-dimensional-gaussian-max}
		Let \(S\subseteq\mathbb R^n\) be an \(r\)-dimensional linear subspace, with orthogonal projection \(P_S\). Suppose that, for some \(\delta\in[0,1]\),
		\[
		\|P_S^\perp X_j\|_2\leq\delta\sqrt n,
		\qquad j=1,\dots,p.
		\]
		Then
		\[
		\mathcal G(X)
		\leq
		\frac{\sigma}{\sqrt n}
		\left\{
		\sqrt r
		+
		\delta\sqrt{2\log(2p)}
		\right\}.
		\]
	\end{proposition}
	
	\begin{proposition}[Lipschitz parametrization]\label{prop:parametric-gaussian-max}
		Let \(f:[0,1]^d\to S^{n-1}\) be \(L\)-Lipschitz with respect to Euclidean distance, and suppose that
		\[
		X_j=\sqrt n\,f(t_j),
		\qquad t_j\in[0,1]^d,
		\qquad j=1,\dots,p.
		\]
		For any integer \(m\geq1\) satisfying \(L\sqrt d/(2m)\leq1\),
		\[
		\mathcal G(X)
		\leq
		\frac{\sigma}{\sqrt n}
		\left\{
		\sqrt{2\log(2m^d)}
		+
		\frac{L\sqrt d}{2m}\sqrt{2\log(2p)}
		\right\}.
		\]
	\end{proposition}
	
	\begin{proposition}[Approximate finite union of low dimensional subspaces]\label{prop:union-subspaces-gaussian-max}
		Let \(S_1,\dots,S_K\subseteq\mathbb R^n\) be linear subspaces with \(\dim(S_k)\leq r\). Write \(u_j=X_j/\sqrt n\), and suppose that for every \(j\) there exists \(k(j)\in\{1,\dots,K\}\) such that
		\[
		\|P_{S_{k(j)}}^\perp u_j\|_2\leq\delta,
		\qquad \delta\in[0,1].
		\]
		Then
		\[
		\mathcal G(X)
		\leq
		\frac{\sigma}{\sqrt n}
		\left\{
		\sqrt r
		+
		\sqrt{2\log K}
		+
		\delta\sqrt{2\log(2p)}
		\right\},
		\]
		where the term \(\sqrt{2\log K}\) is understood to be zero when \(K=1\).
	\end{proposition}
	
	\begin{proposition}[Bounded length geometric path]\label{prop:path-gaussian-max}
		Let \(u_j=X_j/\sqrt n\in S^{n-1}\), and suppose the columns can be ordered so that
		\[
		L_p:=\sum_{j=1}^{p-1}\|u_{j+1}-u_j\|_2<\infty.
		\]
		Then
		\[
		\mathcal G(X)
		\leq
		C\frac{\sigma}{\sqrt n}
		\left\{
		\sqrt{\log\bigl(2(1+L_p)\bigr)}
		+
		1
		\right\},
		\]
		where \(C>0\) is a universal numerical constant. Consequently, if \(L_p=O(1)\), then \(\mathcal G(X)=O(\sigma/\sqrt n)\) independently of \(p\).
	\end{proposition}

	\begin{remark}
		The preceding examples complement existing studies of Lasso prediction under correlated designs, which show that the universal tuning scale can be conservative when the dependence structure of the design is taken into account \citep[see, for example,][]{vandegeer2013lasso}. A related geometric perspective, based on the proximity of covariates to the span of a subset of columns, is discussed in \citet{dalalyan2017prediction}.
		
		Our examples provide concrete geometric descriptions of reductions in stochastic complexity. Clustering, approximate subspaces, and finite unions of subspaces describe different forms of redundancy among design columns, while designs with a Lipschitz parametrization and bounded length paths illustrate how increasing \(p\) through finer discretization need not increase the Gaussian maximum at the independent coordinate rate. The approximation terms quantify the cost of deviations from exact structure, allowing these mechanisms to remain relevant under general position. In the present framework, such geometric properties can improve the lower bounds on the prediction gap by reducing the upper bounds on the stochastic complexity, both for the full design and for localized subsets. These bounds concern geometric redundancy rather than a general effect of pairwise correlation.
	\end{remark}
	
	\subsection{Numerical study}
	\label{subsec:numerical-illustrations}
	
	We examine three aspects of the theory: the effect of design geometry on Gaussian complexity, the informativeness of the eventwise lower bound across Lasso tuning levels, and the role of localization in the expectation bound. We also investigate sensitivity to the second stage penalty and to signal strength. Throughout, the noise has distribution \(N(0,I_n)\), and the design columns satisfy \(\|X_j\|_2^2=n\). Define
	\[
	U_p:=\sqrt{\frac{2\log(2p)}{n}},\qquad
	\mathcal G(X):=\mathbb E\|X^\top\varepsilon/n\|_\infty.
	\]
	The quantity \(U_p\) is a universal upper bound on \(\mathcal G(X)\), not the expected maximum for independent Gaussian coordinates. All Monte Carlo experiments are conditional on the generated design matrices, which are held fixed across noise replications. Standard errors (SE) quantify Monte Carlo uncertainty conditional on these designs.
	
	\paragraph{Gaussian complexity across geometric designs.}
	We take \(n=300\) and \(p=1200\), for which \(U_p=0.2278\). The baseline design consists of independently generated Gaussian columns, subsequently normalized. The remaining designs instantiate the five structures in Section~\ref{subsec:concrete-gaussian-complexity}. Clustered columns are formed by perturbing ten orthonormal representative directions and normalizing the resulting columns. The approximate subspace design uses random unit directions in a subspace of dimension three with orthogonal perturbations of relative size \(0.08\). The Lipschitz design discretizes a circular arc of length \(\pi/2\) at equally spaced parameter values. The union of subspaces uses three mutually orthogonal subspaces of dimension two, with random directions and orthogonal perturbations of relative size \(0.08\). For the path design, we normalize the vectors
	\[
	v(t)=q_1\cos(\pi t/2)+q_2\sin(\pi t/2)+0.15q_3\sin(2\pi t),
	\qquad 0\leq t\leq1,
	\]
	where \(q_1,q_2,q_3\) are orthonormal and the parameter values are equally spaced. The Lipschitz and path designs illustrate Gaussian complexity only; these exact low dimensional constructions need not satisfy the general position condition used for Lasso selection.
	
	For each design, we estimate \(\mathcal G(X)\) from 4000 independent Gaussian noise draws. Table~\ref{tab:numerical-complexity} reports the estimated maximum, its SE, and the corresponding structural upper bound \(B_X\). For the Lipschitz design, we minimize the explicit covering bound over the integer grid \(1\leq m\leq\lceil\sqrt p\rceil\) satisfying the proposition's condition. For the path design, we use the following direct bound rather than assigning a numerical value to the unspecified constant in Proposition~\ref{prop:path-gaussian-max}. With \(g\sim N(0,I_n)\) and \(u_j=X_j/\sqrt n\), the triangle inequality gives
	\[
	\max_{j\leq p}|\langle g,u_j\rangle|
	\leq |\langle g,u_1\rangle|+\sum_{j=1}^{p-1}|\langle g,u_{j+1}-u_j\rangle|.
	\]
	Taking expectations yields the explicit bound
	\[
	\mathcal G(X)\leq \frac{1}{\sqrt n}\sqrt{\frac{2}{\pi}}(1+L_p).
	\]
	This bound is useful for the finite path considered here, whereas Proposition~\ref{prop:path-gaussian-max} provides a sharper dependence on path length for large \(L_p\).
	
	\begin{table}[t]
		\centering\small
		\setlength{\tabcolsep}{4pt}
		\caption{Gaussian complexity for six fixed designs with \(n=300\), \(p=1200\), and \(U_p=0.2278\). Each estimate uses 4000 Gaussian draws. Emp. and Bound denote \(100\widehat{\mathcal G}/U_p\) and \(100B_X/U_p\), respectively; the structural bounds are displayed without taking their minimum with \(U_p\).}
		\label{tab:numerical-complexity}
		\begin{adjustbox}{max width=\linewidth}
			\begin{tabular}{llrrrrr}
				\toprule
				Geometry & Setting & \(\widehat{\mathcal G}\) & SE & \(B_X\) & Emp. (\%) & Bound (\%) \\
				\midrule
				Gaussian columns & Normalized & 0.1988 & 0.0003 & 0.2278 & 87.3 & 100.0 \\
				Clustered & \(K=10,\ \delta=0.08\) & 0.1209 & 0.0005 & 0.1596 & 53.1 & 70.0 \\
				Approx. subspace & \(r=3,\ \delta=0.08\) & 0.0993 & 0.0006 & 0.1182 & 43.6 & 51.9 \\
				Lipschitz design & \(d=1,\ m=7\) & 0.0691 & 0.0006 & 0.1582 & 30.4 & 69.4 \\
				Union of subspaces & \(K=3,\ r=2,\ \delta=0.08\) & 0.1143 & 0.0005 & 0.1855 & 50.2 & 81.4 \\
				Bounded length path & \(L_p=1.69\) & 0.0709 & 0.0006 & 0.1241 & 31.1 & 54.5 \\
				\bottomrule
			\end{tabular}
		\end{adjustbox}
	\end{table}
	
	The structured designs have estimated Gaussian maxima between \(30.4\%\) and \(53.1\%\) of the universal upper bound, compared with \(87.3\%\) for the Gaussian column baseline. All five structural upper bounds are also below \(U_p\). The approximate subspace bound is relatively close to its empirical counterpart, while the Lipschitz covering bound is more conservative. These differences illustrate both the reduction in stochastic complexity due to geometric overlap and the varying sharpness of the available upper bounds.
	
	\paragraph{Prediction improvement and tuning.}
	We next consider two regression designs with \(n=250\), \(p=500\), and \(U_p=0.2351\). The clustered design has 25 orthonormal representative directions and perturbation size \(0.12\); the other design has independently generated and subsequently normalized Gaussian columns. In both designs, the true coefficient vector has six nonzero entries at indices \(1,6,11,16,21,25\), with alternating values \(2\) and \(-2\). These indices belong to distinct clusters in the clustered design. We first estimate \(\mathcal G(X)\) using 5000 independent Gaussian draws for each fixed design, obtaining \(0.1549\) for the clustered design and \(0.2033\) for the Gaussian design. We then consider
	\[
	\lambda_{\mathrm L}=a\widehat{\mathcal G}(X),\qquad
	a\in\{0.75,1,1.25,1.5,2,2.5,3\},
	\]
	together with the common absolute penalty \(\lambda_{\mathrm L}=U_p\). The calibration draws are separate from the 1000 independent regression noise replications per design. Within a replication, the same response is used for all tuning choices and comparisons. For nonempty \(E\), we set \(\lambda_{\mathrm R}=4\|\Sigma_{n,E}\|_2\), following Corollary~\ref{cor:hp-delta-lower-spectral}; for empty \(E\), the correction is zero. The multiplier \(4\) is a common reference value. By Remark~\ref{rem:optimal-spectral-c}, the spectral lower bound coefficient on the event with noise ratio \(\tau\) is maximized at \(c^\star=\tau^{1/3}/(1-\tau^{1/3})\) for \(0<\tau<1\). At the conventional noise ratio \(\tau=1/2\), this maximizer is approximately \(3.85\). This choice is not intended to maximize actual prediction improvement.
	
	We report the mean prediction losses \(\|X(\hat\beta-\beta_0)\|_2^2/n\) for Lasso and its centered ridge refinement, together with the mean prediction gain and its Monte Carlo SE. Let LB denote the eventwise lower bound in Theorem~\ref{thm:eventwise-delta-lower}, evaluated using the realized noise and selected set. A positive LB certifies a positive prediction gain. Its empirical frequency measures the informativeness of the bound, rather than the frequency of positive realized gains. The selected noise ratio, reported as \(100\overline{Z(E)}/\overline{Z([p])}\), where the bars denote Monte Carlo averages, measures the stochastic scale retained after selection relative to that of the full design.
	
	\begin{table}[t]
		\centering\small
		\setlength{\tabcolsep}{3pt}
		\caption{Prediction loss and eventwise certificates for \(\lambda_{\mathrm R}=4\|\Sigma_{n,E}\|_2\), based on 1000 replications per design. The tuning column gives \(a\) for complexity calibrated tuning and \(U_p\) for universal tuning. Lasso and Ridge denote mean prediction losses for Lasso and centered ridge, respectively. The mean gain is followed by its SE in parentheses. Cert. is the percentage with positive LB, and Selected noise is the selected noise ratio in percent.}
		\label{tab:numerical-refinement}
		\begin{adjustbox}{max width=\linewidth}
			\begin{tabular}{lcrrrrrrr}
				\toprule
				Design & Tuning & \(\lambda_{\mathrm L}/U_p\) & Lasso & Ridge & Mean \(\Delta\) (SE) & Mean LB & Cert. (\%) & Selected noise (\%) \\
				\midrule
				Clustered & 0.75 & 0.494 & 0.1631 & 0.1451 & 0.0180 (0.00018) & -0.0170 & 3.7 & 96.8 \\
				& 1.00 & 0.659 & 0.2230 & 0.1896 & 0.0334 (0.00026) & -0.0074 & 33.9 & 86.4 \\
				& 1.25 & 0.823 & 0.3032 & 0.2508 & 0.0524 (0.00035) & 0.0076 & 73.9 & 75.9 \\
				& 1.50 & 0.988 & 0.4024 & 0.3270 & 0.0754 (0.00045) & 0.0240 & 94.3 & 71.7 \\
				& \(U_p\) & 1.000 & 0.4103 & 0.3331 & 0.0771 (0.00046) & 0.0252 & 94.7 & 71.6 \\
				& 2.00 & 1.318 & 0.6563 & 0.5233 & 0.1330 (0.00071) & 0.0635 & 99.3 & 71.0 \\
				& 2.50 & 1.647 & 0.9830 & 0.7774 & 0.2056 (0.00102) & 0.1168 & 99.9 & 70.8 \\
				& 3.00 & 1.976 & 1.3810 & 1.0863 & 0.2947 (0.00143) & 0.1849 & 100.0 & 70.6 \\
				\midrule
				Gaussian & 0.75 & 0.648 & 0.1677 & 0.1393 & 0.0284 (0.00032) & -0.0719 & 0.0 & 99.2 \\
				& 1.00 & 0.865 & 0.2636 & 0.1921 & 0.0715 (0.00036) & -0.0222 & 19.6 & 88.2 \\
				& \(U_p\) & 1.000 & 0.3431 & 0.2445 & 0.0986 (0.00040) & 0.0125 & 62.7 & 71.2 \\
				& 1.25 & 1.081 & 0.3966 & 0.2808 & 0.1157 (0.00042) & 0.0320 & 80.3 & 63.6 \\
				& 1.50 & 1.297 & 0.5604 & 0.3933 & 0.1671 (0.00051) & 0.0798 & 98.1 & 53.9 \\
				& 2.00 & 1.729 & 0.9782 & 0.6809 & 0.2973 (0.00067) & 0.1806 & 100.0 & 52.2 \\
				& 2.50 & 2.162 & 1.5154 & 1.0508 & 0.4646 (0.00084) & 0.3141 & 100.0 & 52.2 \\
				& 3.00 & 2.594 & 2.1723 & 1.5031 & 0.6692 (0.00101) & 0.4831 & 100.0 & 52.2 \\
				\bottomrule
			\end{tabular}
		\end{adjustbox}
	\end{table}
	
	Centered ridge reduces the Lasso prediction loss in all 1000 replications at each tuning level in both designs. The mean gain increases with the Lasso penalty, as does the frequency of positive certificates. For example, at \(a=1.25\) in the clustered design the penalty is only \(82.3\%\) of \(U_p\), yet the mean LB is positive and \(73.9\%\) of replications have a positive certificate. At the common universal penalty, the clustered design has a certificate frequency of \(94.7\%\), compared with \(62.7\%\) for the Gaussian design. This provides a concrete example of prediction suboptimality below the conventional universal benchmark and illustrates why the absolute penalty alone does not determine the informativeness of the bound.
	
	The selected noise ratio decreases with the tuning level in both designs. At \(a=2\), for example, the ratios are \(71.0\%\) and \(52.2\%\) in the clustered and Gaussian designs, respectively. The corresponding average selected model sizes are \(18.46\) and \(6.00\), and the average largest selected Gram eigenvalues are \(5.16\) and \(1.26\). Thus the selected stochastic scale and the selected Gram geometry can differ substantially across designs. The latter also affects the magnitude of the certified gain. The smaller selected noise maximum illustrates the benefit of controlling noise on the selected set rather than on the full design; deterministic localization is examined separately below. A negative LB does not preclude improvement, as illustrated by the positive realized gains at smaller penalties. The comparison with centered ridge is therefore used to identify removable prediction error, rather than to rank refitting procedures in general.
	
	\paragraph{Localization and the expectation lower bound.}
	We now examine Theorem~\ref{thm:lasso-dmse-lower-bound} using the same regression responses but the distinct ridge choice \(\lambda_{\mathrm R}=4|E|\) when \(E\neq\emptyset\). We retain the same reference multiplier for comparability, not as an optimal choice for mean prediction improvement. We consider the full design \([p]\), the true support \(S_0\), and, for the clustered design, the union \(T_C\) of the six clusters containing the true variables. These are deterministic sets conditional on the design and true coefficient vector; the latter two are oracle localization sets, not practical screening rules. The Gaussian maxima \(\nu(T_0)\) and the full design second moment are estimated using 5000 Gaussian draws per set, independently of the regression noise. Containment and nonempty selection probabilities are estimated by their empirical frequencies.
	
	For a fixed \(T_0\), let
	\[
	\widehat R(T_0):=\{\widehat{\mathbb E}Z([p])^2\}^{1/2}
	\{1-\widehat{\mathbb P}(E\subseteq T_0)\}^{1/2}.
	\]
	Here \(\widehat{\mathbb E}Z([p])^2\) is the Monte Carlo estimate of \(\mathbb E[Z([p])^2]\). The corresponding plug in estimate of the expectation lower bound is
	\[
	\widehat B(T_0)=\frac{\lambda_{\mathrm L}}{2}
	\left\{\frac{18}{25}\lambda_{\mathrm L}\widehat{\mathbb P}(E\neq\emptyset)
	-\widehat\nu(T_0)-\widehat R(T_0)\right\}.
	\]
	All main regression fits have nonempty selected models. Table~\ref{tab:numerical-localization} reports selected tuning levels, including the universal benchmark. The empirical DMSE uses the same ridge rule as the expectation theorem. The estimates \(\widehat B(T_0)\) are Monte Carlo estimates of population lower bounds, not confidence bounds for DMSE.
	
	\begin{table}[t]
		\centering\small
		\setlength{\tabcolsep}{3pt}
		\caption{Localization and the expectation lower bound for \(\lambda_{\mathrm R}=4|E|\). Cont. denotes the percentage of replications with \(E\subseteq T_0\). The empirical DMSE is common to all localization sets for a given design and tuning level. Values in parentheses are Monte Carlo SEs.}
		\label{tab:numerical-localization}
		\begin{adjustbox}{max width=\linewidth}
			\begin{tabular}{lclrrrrr}
				\toprule
				Design & Tuning & \(T_0\) & \(\widehat\nu(T_0)\) & Cont. (\%) & \(\widehat R(T_0)\) & \(\widehat B(T_0)\) & Empirical DMSE (SE) \\
				\midrule
				Clustered & 1.50 & \([p]\) & 0.1549 & 100.0 & 0.0000 & 0.0014 & 0.0241 (0.00010) \\
				&  & \(T_C\) & 0.1178 & 99.4 & 0.0122 & 0.0043 &  \\
				&  & \(S_0\) & 0.1042 & 0.0 & 0.1574 & -0.0110 &  \\
				\addlinespace[2pt]
				& \(U_p\) & \([p]\) & 0.1549 & 100.0 & 0.0000 & 0.0017 & 0.0247 (0.00010) \\
				&  & \(T_C\) & 0.1178 & 99.5 & 0.0111 & 0.0047 &  \\
				&  & \(S_0\) & 0.1042 & 0.0 & 0.1574 & -0.0109 &  \\
				\addlinespace[2pt]
				& 2.00 & \([p]\) & 0.1549 & 100.0 & 0.0000 & 0.0106 & 0.0434 (0.00013) \\
				&  & \(T_C\) & 0.1178 & 100.0 & 0.0000 & 0.0163 &  \\
				&  & \(S_0\) & 0.1042 & 0.0 & 0.1574 & -0.0060 &  \\
				\midrule
				Gaussian & \(U_p\) & \([p]\) & 0.2033 & 100.0 & 0.0000 & -0.0040 & 0.0244 (0.00011) \\
				&  & \(S_0\) & 0.1047 & 59.3 & 0.1305 & -0.0078 &  \\
				\addlinespace[2pt]
				& 1.50 & \([p]\) & 0.2033 & 100.0 & 0.0000 & 0.0025 & 0.0434 (0.00012) \\
				&  & \(S_0\) & 0.1047 & 96.8 & 0.0366 & 0.0119 &  \\
				\addlinespace[2pt]
				& 2.00 & \([p]\) & 0.2033 & 100.0 & 0.0000 & 0.0182 & 0.0774 (0.00016) \\
				&  & \(S_0\) & 0.1047 & 100.0 & 0.0000 & 0.0382 &  \\
				\bottomrule
			\end{tabular}
		\end{adjustbox}
	\end{table}
	
	The results show why localization must be considered jointly with containment. In the clustered design, \(\widehat\nu(S_0)=0.1042\) is smaller than \(\widehat\nu(T_C)=0.1178\) and \(\widehat\nu([p])=0.1549\), but none of the selected models is contained in \(S_0\) at the tuning levels examined. The resulting containment penalty makes the true support bound less informative. In contrast, at \(a=1.5\), the cluster union contains the selected set in \(99.4\%\) of replications and increases the estimated lower bound from \(0.00144\) to \(0.00433\). At the common universal penalty, the corresponding estimates are \(0.00169\) and \(0.00474\), while the empirical DMSE is \(0.02474\). This illustrates the value of coarse localization when exact support containment is unlikely.
	
	For the Gaussian design, at \(a=1.5\), true support containment occurs in \(96.8\%\) of replications and increases the estimated bound from \(0.00248\) to \(0.01192\). At \(a=2\), the empirical containment frequency is \(100\%\), and the plug in bounds are \(0.01818\) for the full design and \(0.03821\) for the true support. Even when all 1000 replications satisfy containment, the one sided 95\% exact binomial lower confidence limit for the containment probability is \(99.70\%\), rather than one. These results illustrate the tradeoff between reducing the localized noise maximum and maintaining reliable containment. At the universal penalty in the Gaussian design, both plug in bounds are negative despite positive empirical DMSE, again reflecting the conservativeness of sufficient conditions.
	
	\paragraph{Sensitivity to ridge tuning and signal strength.}
	Table~\ref{tab:numerical-sensitivity} compares \(\lambda_{\mathrm R}=c\|\Sigma_{n,E}\|_2\) for five choices of \(c\) at \(a=1.5\). Increasing \(c\) reduces the average realized gain in both designs while increasing the frequency of positive certificates over the displayed grid. The mean LB is largest among the displayed choices at \(c=4\) for the clustered design and \(c=2\) for the Gaussian design. Thus the choice that yields the largest mean certified gain need not yield the largest actual gain or the highest certificate frequency.
	
	\begin{table}[t]
		\centering\small
		\setlength{\tabcolsep}{4pt}
		\caption{Sensitivity to \(c\) at \(a=1.5\) for the ridge choice \(\lambda_{\mathrm R}=c\|\Sigma_{n,E}\|_2\), using the main regression replications. Cert. denotes the percentage with positive eventwise LB.}
		\label{tab:numerical-sensitivity}
		\begin{adjustbox}{max width=\linewidth}
			\begin{tabular}{r rrr rrr}
				\toprule
				& \multicolumn{3}{c}{Clustered} & \multicolumn{3}{c}{Gaussian} \\
				\cmidrule(lr){2-4}\cmidrule(lr){5-7}
				\(c\) & Mean \(\Delta\) & Mean LB & Cert. (\%) & Mean \(\Delta\) & Mean LB & Cert. (\%) \\
				\midrule
				0.5 & 0.2310 & -0.2089 & 1.6 & 0.4581 & -0.2457 & 13.1 \\
				1.0 & 0.1794 & -0.0429 & 26.0 & 0.3737 & 0.0129 & 61.2 \\
				2.0 & 0.1231 & 0.0149 & 72.5 & 0.2664 & 0.0866 & 92.9 \\
				4.0 & 0.0754 & 0.0240 & 94.3 & 0.1671 & 0.0798 & 98.1 \\
				8.0 & 0.0424 & 0.0180 & 97.4 & 0.0953 & 0.0532 & 99.4 \\
				\bottomrule
			\end{tabular}
		\end{adjustbox}
	\end{table}
	
	We also repeat the regression experiment with nonzero coefficient magnitudes reduced to \(0.25\) and \(0.75\), retaining the same support and design matrices and using 500 independent replications per design and signal setting. We consider \(a\in\{1,1.5,2\}\) and \(\lambda_{\mathrm R}=4\|\Sigma_{n,E}\|_2\). Table~\ref{tab:numerical-signal} summarizes the results. At magnitude \(0.75\), all observed gains are positive. At magnitude \(0.25\), the mean gains remain positive, but larger penalties increasingly lead to empty selection and hence zero gain. At \(a=2\), for example, selection occurs in \(77.6\%\) and \(7.6\%\) of replications in the clustered and Gaussian designs, respectively, with positive gains in precisely those replications. A single negative gain occurs in the Gaussian design at \(a=1.5\). These results illustrate the role of nonempty selection and the limits of refinement when the signal is small relative to the initial penalty.
	
	\begin{table}[t]
		\centering\small
		\setlength{\tabcolsep}{4pt}
		\caption{Signal strength sensitivity, with 500 replications per design and signal setting. Signal is the absolute nonzero coefficient magnitude. Pos., Neg., and Cert. denote the percentages of positive gains, negative gains, and positive eventwise bounds. Nonempty denotes the percentage with nonempty selection.}
		\label{tab:numerical-signal}
		\begin{adjustbox}{max width=\linewidth}
			\begin{tabular}{rlrrrrrr}
				\toprule
				Signal & Design & \(a\) & Mean \(\Delta\) & Pos. (\%) & Neg. (\%) & Cert. (\%) & Nonempty (\%) \\
				\midrule
				0.25 & Clustered & 1.0 & 0.028944 & 100.0 & 0.0 & 43.4 & 100.0 \\
				& Clustered & 1.5 & 0.052525 & 100.0 & 0.0 & 95.2 & 100.0 \\
				& Clustered & 2.0 & 0.030398 & 77.6 & 0.0 & 77.4 & 77.6 \\
				& Gaussian & 1.0 & 0.050370 & 100.0 & 0.0 & 25.0 & 100.0 \\
				& Gaussian & 1.5 & 0.037000 & 83.8 & 0.2 & 82.8 & 84.0 \\
				& Gaussian & 2.0 & 0.002694 & 7.6 & 0.0 & 7.6 & 7.6 \\
				\midrule
				0.75 & Clustered & 1.0 & 0.033928 & 100.0 & 0.0 & 38.4 & 100.0 \\
				& Clustered & 1.5 & 0.078810 & 100.0 & 0.0 & 93.2 & 100.0 \\
				& Clustered & 2.0 & 0.138851 & 100.0 & 0.0 & 99.8 & 100.0 \\
				& Gaussian & 1.0 & 0.072023 & 100.0 & 0.0 & 25.2 & 100.0 \\
				& Gaussian & 1.5 & 0.167539 & 100.0 & 0.0 & 97.2 & 100.0 \\
				& Gaussian & 2.0 & 0.297852 & 100.0 & 0.0 & 100.0 & 100.0 \\
				\bottomrule
			\end{tabular}
		\end{adjustbox}
	\end{table}
	
	\paragraph{Numerical accuracy.}
	We use \texttt{glmnet} without an intercept or additional standardization, with convergence threshold \(10^{-16}\) and at most \(10^6\) passes. The returned nonzero coefficient set represents \(E\), since the Lasso support and equicorrelation set agree almost surely under general position and Gaussian noise. No coefficient truncation or polishing is applied. Across the 22000 main and signal sensitivity fits, the support agrees with the numerical equicorrelation set (relative tolerance \(10^{-6}\)) in \(99.95\%\) of each experiment. The maximum relative KKT residual is \(1.11631\times10^{-6}\); two fits exceed \(10^{-6}\), but all satisfy the final diagnostic tolerance of \(2\times10^{-6}\), adopted after the initial run without refitting. The ridge correction uses the actual residual gradient, while the theoretical certificates use the exact KKT expressions. No eventwise bound violations are detected at tolerance \(10^{-6}(1+|\Delta|)\) across all ridge settings, and the maximum absolute prediction identity discrepancy is \(6.66\times10^{-16}\). The expectation bounds are population inequalities; their reported values are Monte Carlo estimates.
	
	\subsection{Complementary results}
	\label{subsec:complementary-results}
	
	The preceding results and numerical study identify the selected noise level and design geometry as central to prediction suboptimality. We now record three complementary consequences. The first shows when the correction preserves the selected signs and permits a sharper eventwise criterion. The second replaces control of a localized maximum by stability of a signed model. The third illustrates how the expectation bound extends beyond Gaussian noise. These results are not required for the main comparison and do not impose additional assumptions on the preceding theorems.
	
	\subsubsection{Sign preservation and a refined eventwise criterion}
	\label{subsubsec:sign-preservation}
	
	For nonempty \(E\), write \(A_E:=(\Sigma_{n,E}+\lambda_{\mathrm R}I_{|E|})^{-1}\). Stronger ridge regularization can ensure that the correction has the same signs as the equicorrelation vector. A Neumann series argument gives the following sufficient condition.
	
	\begin{corollary}[Sign preserving correction and refined eventwise bound]
		\label{cor:sign-aware-delta-lower}
		Suppose \(E\neq\emptyset\) and
		\begin{equation*}
			\lambda_{\mathrm R}>2\|\Sigma_{n,E}\|_\infty.
		\end{equation*}
		Then \(\operatorname{sign}(\hat\delta_E)=s\), and the correction satisfies the exact identity
		\[
		\|\hat\delta_E\|_1
		=\lambda_{\mathrm L}s^\top A_Es.
		\]
		In particular,
		\begin{equation}
			\Delta(\hat\beta_{\mathrm L},\hat\beta_{\mathrm R})
			\geq\lambda_{\mathrm L}^2
			s^\top(A_E+\lambda_{\mathrm R}A_E^2)s
			-2\lambda_{\mathrm L}Z(E)s^\top A_Es.
			\label{eq:sign-aware-exact-eventwise}
		\end{equation}
		On \(\mathcal E_\tau(E)\), this guarantees \(\Delta>0\) whenever
		\begin{equation}
			\tau<\tau_{\mathrm N}(E,s)
			:=\frac12\left(1+\lambda_{\mathrm R}
			\frac{s^\top A_E^2s}{s^\top A_Es}\right).
			\label{eq:sign-aware-positive-criterion}
		\end{equation}
		Under the same condition,
		\(\tau_{\mathrm N}(E,s)>5/6\).
		
		The proof, including an additional explicit bound in terms of the sign direction Rayleigh quotient, is given in Appendix~\ref{appA}.
	\end{corollary}
	
	\begin{remark}[Role of sign preservation]
		The sign condition ensures that the correction moves in the selected sign direction coordinatewise. In particular, it preserves the signs of nonzero Lasso coefficients; it does not establish recovery of the true signs. The exact expression for \(\|\hat\delta_E\|_1\), together with the exact quadratic term, gives a lower bound no smaller than that in Theorem~\ref{thm:eventwise-delta-lower} at the same ridge penalty. The refinement is meaningful at a given ridge penalty, rather than in terms of the existence of a penalty yielding a positive lower bound. Indeed, under the spectral choice \(\lambda_{\mathrm R}=cM_E\), whenever the row sum condition is satisfied, the inequality \(s^\top A_E^2s/(s^\top A_Es)\geq(\lambda_{\mathrm R}+M_E)^{-1}\) gives
		\[
		\tau_{\mathrm N}(E,s)
		\geq\frac{1+2c}{2(1+c)}
		>\frac{c(1+2c)}{2(1+c)^2}.
		\]
		The last expression is the positivity threshold of the lower bound in Corollary~\ref{cor:hp-delta-lower-spectral}. Thus, the sign dependent bound guarantees prediction improvement over a strictly larger range of noise ratios at the same ridge penalty. The row sum condition is not necessary for prediction improvement, and the result does not identify an optimal ridge penalty.
	\end{remark}
	
	\subsubsection{Signed model stability}
	\label{subsubsec:signed-model-stability}
	
	The general expectation theorem also simplifies when the selected signed model concentrates around a deterministic target. Here the signed model error controls the stochastic cross term directly. Unlike support containment, this condition concerns both the selected set and its signs.
	
	\begin{theorem}[Signed model stability specialization]\label{thm:lasso-accurate-selection}
		Let the design be fixed, let \(\varepsilon \sim \mathcal{N}(0,\sigma^2 I_n)\), let \(S_0\) denote the true support, and set \(s_0\coloneqq\operatorname{sign}(\beta_{0,S_0})\). Choose \(\lambda_{\mathrm{R}}=c|E|\) on \(E\neq\emptyset\), where \(c>0\), and use the zero correction otherwise. Then
		\begin{align*}
			\mathrm{DMSE}	
			&\geq\frac{\lambda_{\mathrm{L}}}{\sqrt{c}}\Biggl(\frac{\sqrt{c}(2c+1)}{(c+1)^2}\lambda_{\mathrm{L}}\mathbb{P}(E\neq \emptyset)\\
			&\qquad-\sigma\left(1+\frac{1}{\sqrt{n}}\right)\sqrt{\mathbb{P}((E,s)\ne(S_0,s_0))}\Biggr).
		\end{align*}
	\end{theorem}
	
	\begin{remark}[Signed model stability]
		Under suitable conditions on the design, sparsity, signal strength, and tuning, including the strong irrepresentable condition, the signed model error probability \(\mathbb P((E,s)\neq(S_0,s_0))\) can decay rapidly; see \cite[Theorem~4]{zhao2006model}. Under general position and Gaussian noise, the Lasso support and equicorrelation set agree almost surely, so coefficient sign recovery transfers to recovery of \((E,s)\). More generally, the proof applies to any fixed deterministic reference pair \((S^\star,s^\star)\), provided \(X_{S^\star}\) has full column rank when \(S^\star\neq\emptyset\), with the stated empty set conventions otherwise. Whenever the positive term in the corresponding expectation bound dominates the signed model stability error, the refinement improves prediction in expectation. The proof is given in Appendix~\ref{appA}.
	\end{remark}
	
	\subsubsection{Martingale difference noise}
	\label{subsec:martingale-noise}
	
	The expectation lower bound is not intrinsically Gaussian. For martingale difference errors, a finite dimensional martingale maximal inequality controls the localized noise level.
	
	\begin{assumption}[Martingale difference noise]
		\label{ass:martingale-noise}
		Let \((\mathcal F_i)_{i=0}^n\) be a filtration. Suppose that \(\varepsilon_i\) is \(\mathcal F_i\)-measurable and
		\[
		\mathbb E[\varepsilon_i\mid\mathcal F_{i-1}]=0,
		\qquad
		\mathbb E[\varepsilon_i^2\mid\mathcal F_{i-1}]\leq\sigma^2
		\quad\text{a.s.},\qquad i=1,\dots,n.
		\]
		The design \(X=(x_{ij})\) is deterministic and, in addition to column normalization, satisfies
		\[
		\max_{1\leq i\leq n,\;1\leq j\leq p}|x_{ij}|\leq K_X
		\]
		for some finite constant \(K_X>0\).
	\end{assumption}
	
	\begin{proposition}[Localized martingale maximal inequality]
		\label{prop:martingale-localized-max}
		Under Assumption~\ref{ass:martingale-noise}, for every deterministic \(T\subseteq\{1,\dots,p\}\),
		\[
		\mathbb E\left\|\frac{X_T^\top\varepsilon}{n}\right\|_\infty
		\leq
		2\sqrt{2}\,K_X\sigma\sqrt{\frac{\log(1+|T|)}{n}}.
		\]
		Equivalently,
		\[
		\nu(T)\leq
		2\sqrt{2}\,K_X\sigma\sqrt{\frac{\log(1+|T|)}{n}}.
		\]
	\end{proposition}

	\begin{corollary}[Global expectation lower bound under martingale noise]
		\label{cor:martingale-global-dmse}
		Suppose that Assumption~\ref{ass:martingale-noise} holds. Choose \(\lambda_{\mathrm R}=c|E|\) on \(E\neq\emptyset\), where \(c>0\), and use the zero correction otherwise. Then
		\[
		\mathrm{DMSE}
		\geq
		\frac{2\lambda_{\mathrm L}}{c}
		\left\{
		\frac{c(2c+1)}{2(c+1)^2}
		\lambda_{\mathrm L}\mathbb P(E\neq\emptyset)
		-
		2\sqrt{2}\,K_X\sigma
		\sqrt{\frac{\log(1+p)}{n}}
		\right\}.
		\]
	\end{corollary}
	
	The proofs of the martingale bounds are given in Appendix~\ref{appA}.

	\begin{remark}[Tuning and localization]
		\label{rem:martingale-tuning-localization}
		Let
		\[
		g(c):=\frac{c(2c+1)}{2(c+1)^2}.
		\]
		If \(\mathbb P(E\neq\emptyset)>0\), Corollary~\ref{cor:martingale-global-dmse} gives a positive lower bound whenever
		\[
		\lambda_{\mathrm L}
		>
		\frac{2\sqrt{2}\,K_X\sigma}
		{g(c)\mathbb P(E\neq\emptyset)}
		\sqrt{\frac{\log(1+p)}{n}}.
		\]
		Thus the usual logarithmic stochastic scale persists under martingale difference noise with bounded conditional second moments and a uniformly bounded deterministic design. Proposition~\ref{prop:martingale-localized-max} also controls the first localized stochastic term in Theorem~\ref{thm:lasso-dmse-lower-bound}, with \(p\) replaced by \(|T_0|\). A fully explicit localized specialization additionally requires control of
		\[
		\mathbb E\left[\|X^\top\varepsilon/n\|_\infty^2\right],
		\]
		which appears in the support containment remainder.
	\end{remark}

	\section{Discussion}
	\label{sec:discussion-extensions}
	
	The results provide a perspective on Lasso tuning that complements the usual analysis through prediction risk upper bounds. A tuning choice that ensures favorable prediction guarantees may still induce substantial removable prediction error. In particular, the prediction gain can be of the same order as the classical sparse prediction rate under suitable selection size and spectral conditions. This phenomenon is compatible with minimax rate optimality. Our results establish sufficient conditions for prediction improvement rather than uniform dominance or necessary conditions for optimal Lasso tuning.
	
	Design geometry enters the analysis through two distinct mechanisms. The localized noise level \(\nu(T_0)\) influences whether a positive prediction gain can be guaranteed, whereas the spectral structure of the selected Gram matrix influences the magnitude of the lower bound. Geometric redundancy can reduce stochastic complexity without necessarily improving the spectral factor. Understanding the interaction between these two mechanisms may provide further insights into prediction and tuning under correlated designs. The second stage tuning choices considered here are motivated primarily by the theoretical bounds. In practice, validation based selection of the ridge parameter is a natural possibility, although its theoretical analysis is beyond the scope of this work.
	
	While the prediction comparison developed here relies on the particular shrinkage structure of the Lasso, the reduction in stochastic complexity through geometric redundancy is a more general phenomenon. A related perspective appears in the descriptive complexity framework for model selection under correlated designs \cite{zhangliu2026descriptive}. Although the underlying complexity measures and theoretical objectives differ, both perspectives suggest that nominal dimension alone may provide an inadequate description of statistical complexity. Developing tuning and model selection principles that more directly account for the geometry and redundancy of candidate predictors remains an interesting direction.
	
	More broadly, dependence among random quantities can substantially reduce their collective fluctuations relative to what cardinality based bounds suggest. Generic chaining \cite{talagrand2021upper} provides a powerful mathematical framework for capturing such effects through the metric geometry of stochastic processes. Further exploiting these mechanisms in statistical modeling may be particularly valuable in finite sample settings, where collecting additional observations can be costly or impractical. The present analysis provides a concrete example in high dimensional linear regression. Related investigations for other \(\ell_1\) regularized methods, such as the Dantzig selector \cite{candes2007dantzig} and Adaptive Lasso \cite{zou2006adaptive}, would require their own active set and optimality arguments.

	\clearpage
	\appendix
	
	\section{Technical Lemmas and Proofs}\label{appA}
	
	We first record two auxiliary lemmas, then prove the expectation lower bound and the complementary results deferred from the main text.
	
	\begin{lemma}\label{lem:gaussian-infty-sq-bound}
		Suppose that \(\varepsilon \sim N(0,\sigma^{2}I_n)\) and 
		\(\|X_j\|_{2}^{2}=n\) for all \(j\). Then
		\(
		\mathbb{E}[\|X^{\top}\varepsilon/n\|_{\infty}^{2}]
		\leq \sigma^{2}(2\log(4p)+1)/n.
		\)
	\end{lemma}
	\begin{proof}
		For each \(j\),
		\begin{align*}
			\mathbb{P}(|X_j^{\top}\varepsilon|\geq t)
			&\leq 2\exp(-t^{2}/(2\sigma^{2}n)).
		\end{align*}
		Define \(M=\|X^{\top}\varepsilon\|_{\infty}\). Then
		\begin{align*}
			\mathbb{P}(M\geq t)
			&\leq 2p\exp(-t^{2}/(2\sigma^{2}n)).
		\end{align*}
		Using \(\mathbb{E}[M^{2}] = \int_{0}^{\infty} 2t\,\mathbb{P}(M\geq t)\,dt\),
		let \(t_{0}=\sigma\sqrt{n}\sqrt{2\log(4p)}\). Then
		\begin{align*}
			\int_{0}^{t_{0}} 2t\,\mathbb{P}(M\geq t)\,dt
			&\leq t_{0}^{2}
			= 2\sigma^{2}n\log(4p), \\
			\int_{t_{0}}^{\infty} 2t\,\mathbb{P}(M\geq t)\,dt
			&\leq 4p \int_{t_{0}}^{\infty} 
			t\exp(-t^{2}/(2\sigma^{2}n))\,dt \\
			&= 4p\,\sigma^{2}n \int_{\log(4p)}^{\infty} e^{-u}\,du \\
			&= \sigma^{2}n .
		\end{align*}
		Combining these two inequalities, we have
		\begin{align*}
			\mathbb{E}[M^{2}]
			&\leq \sigma^{2}n \,(2\log(4p)+1),
		\end{align*}
		which concludes the proof.
	\end{proof}
	
	\begin{lemma}\label{lem:eventE-empty-prob}
		Assume that \(\varepsilon \sim \mathcal N(0,\sigma^2 I_n)\). Then
		\begin{align*}
			\mathbb P(E=\emptyset)
			\leq
			\Phi\!\left(
			\frac{\lambda_{\mathrm L}-\|\Sigma_n \beta_0\|_\infty}{\sigma/\sqrt n}
			\right),
		\end{align*}
		where \(\Sigma_n := X^\top X/n\).
	\end{lemma}
	
	\begin{proof}
		By the KKT conditions, \(\hat\beta_{\mathrm L}=0\) if and only if
		\begin{align*}
			\frac{1}{n}\,\|X^\top y\|_\infty \leq \lambda_{\mathrm L}.
		\end{align*}
		Hence
		\begin{align*}
			\mathbb P(E=\emptyset)
			&=
			\mathbb P(\hat\beta_{\mathrm L}=0) \\
			&=
			\mathbb P\!\left(
			\frac{1}{n}\,\|X^\top y\|_\infty \leq \lambda_{\mathrm L}
			\right).
		\end{align*}
		Let \(j^\star\) satisfy
		\begin{align*}
			|(\Sigma_n \beta_0)_{j^\star}|
			=
			\|\Sigma_n \beta_0\|_\infty .
		\end{align*}
		Then
		\begin{align*}
			\left\{
			\frac{1}{n}\,\|X^\top y\|_\infty \leq \lambda_{\mathrm L}
			\right\}
			\subseteq
			\left\{
			\frac{1}{n}\,|X_{j^\star}^\top y|
			\leq
			\lambda_{\mathrm L}
			\right\}.
		\end{align*}
		Write
		\begin{align*}
			m
			:=
			(\Sigma_n \beta_0)_{j^\star}
			=
			\frac{1}{n} X_{j^\star}^\top X \beta_0,
			\qquad
			|m|
			=
			\|\Sigma_n \beta_0\|_\infty .
		\end{align*}
		Using \(y=X\beta_0+\varepsilon\), we obtain
		\begin{align*}
			\mathbb P(E=\emptyset)
			\leq
			\mathbb P\!\left(
			\left|
			\frac{1}{n} X_{j^\star}^\top \varepsilon + m
			\right|
			\leq
			\lambda_{\mathrm L}
			\right).
		\end{align*}
		Let
		\begin{align*}
			Z
			:=
			\frac{1}{n} X_{j^\star}^\top \varepsilon .
		\end{align*}
		Since \(Z \sim N(0,\sigma^2/n)\) is symmetric about zero, we have
		\(Z \stackrel{d}{=} -Z\).
		Consequently,
		\begin{align*}
			\mathbb P(|Z+m|\leq \lambda_{\mathrm L})
			&=
			\mathbb P(|-Z+m|\leq \lambda_{\mathrm L}) \\
			&=
			\mathbb P(|Z-m|\leq \lambda_{\mathrm L}).
		\end{align*}
		If \(m\geq 0\), then \(|Z-m| = |Z-|m||\); if \(m<0\), then \(-m=|m|\) and
		\(|Z-m| = |Z+|m||\).
		In either case,
		\begin{align*}
			\mathbb P(|Z+m|\leq \lambda_{\mathrm L})
			=
			\mathbb P(|Z+|m||\leq \lambda_{\mathrm L}).
		\end{align*}
		Because \(|m|\geq 0\), the event
		\(\{|Z+|m||\leq \lambda_{\mathrm L}\}\)
		is contained in
		\(\{Z\leq \lambda_{\mathrm L}-|m|\}\),
		and therefore
		\begin{align*}
			\mathbb P(E=\emptyset)
			\leq
			\mathbb P\!\left(
			Z \leq \lambda_{\mathrm L}-|m|
			\right)
			=
			\Phi\!\left(
			\frac{\lambda_{\mathrm L}-\|\Sigma_n \beta_0\|_\infty}{\sigma/\sqrt n}
			\right).
		\end{align*}
	\end{proof}
	
	\begin{remark}\label{rem:eventE-empty-exp-decay}
		In particular, if \(\|\Sigma_n \beta_0\|_\infty \geq \lambda_{\mathrm L} + t\,\sigma/\sqrt n\)
		for some \(t>0\), then
		\[
		\mathbb P(E=\emptyset) \leq \Phi(-t),
		\]
		which decays exponentially in \(t^2\). Hence, under the usual scaling in which \(\lambda_{\mathrm L}\to 0\) while \(\|\Sigma_n \beta_0\|_\infty\) remains of constant order, the margin \(\|\Sigma_n \beta_0\|_\infty - \lambda_{\mathrm L}\) is bounded away from zero, and consequently the probability that the Lasso solution is exactly zero is exponentially small.
	\end{remark}
	
	\subsection*{Proof of Theorem~\ref{thm:lasso-dmse-lower-bound}}
	\begin{proof}
		For \(E\neq\emptyset\), put \(m=|E|\) and choose \(\lambda_{\mathrm R}=cm\). By column normalization, the largest eigenvalue of \(\Sigma_{n,E}\) is at most \(\operatorname{tr}(\Sigma_{n,E})=m\). The deterministic term in \eqref{eq:3}, bounded as in Theorem~\ref{thm:eventwise-delta-lower}, therefore satisfies
		\[
		H_{E,s}(\lambda_{\mathrm L},\lambda_{\mathrm R})
		\geq\lambda_{\mathrm L}^2\frac{2c+1}{(c+1)^2}.
		\]
		Moreover, the same theorem gives
		\(\|\hat\delta_E\|_1\leq\lambda_{\mathrm L}m/\lambda_{\mathrm R}=\lambda_{\mathrm L}/c\).
		With the zero correction when \(E=\emptyset\), taking expectations in \eqref{eq:3} yields
		\begin{equation}
			\mathrm{DMSE}\geq
			\lambda_{\mathrm L}^2\frac{2c+1}{(c+1)^2}
			\mathbb P(E\neq\emptyset)
			-\frac{2\lambda_{\mathrm L}}c\mathbb E[Z(E)].
			\label{eq:dmse-final-localized-lower-bound}
		\end{equation}
		For every deterministic \(T_0\subseteq[p]\),
		\[
		Z(E)\leq\|X_{T_0}^\top\varepsilon/n\|_\infty
		+\|X^\top\varepsilon/n\|_\infty
		\mathbf1_{\{E\not\subseteq T_0\}}.
		\]
		Taking expectations and applying Cauchy--Schwarz to the second term gives
		\begin{align*}
			\mathbb E[Z(E)]\leq{}&\nu(T_0)
			+\left(\mathbb E\|X^\top\varepsilon/n\|_\infty^2\right)^{1/2}
			\sqrt{\mathbb P(E\not\subseteq T_0)}.
		\end{align*}
		Substitution into \eqref{eq:dmse-final-localized-lower-bound} and optimization over \(T_0\) prove the result.
	\end{proof}
	
	\subsection*{Proof of Corollary~\ref{cor:sign-aware-delta-lower}}
	\begin{proof}
		Suppose \(E\neq\emptyset\), and put \(g_E=X_E^\top\varepsilon/n\).
		The exact prediction gap decomposition~\eqref{eq:3}, together with
		\(\hat\delta_E=\lambda_{\mathrm L}A_Es\), gives
		\begin{align}
			\Delta(\hat\beta_{\mathrm L},\hat\beta_{\mathrm R})
			&=\lambda_{\mathrm L}^2
			s^\top(2A_E-A_E\Sigma_{n,E}A_E)s
			-2\lambda_{\mathrm L}(A_Es)^\top g_E\notag\\
			&=\lambda_{\mathrm L}^2s^\top(A_E+\lambda_{\mathrm R}A_E^2)s
			-2\lambda_{\mathrm L}(A_Es)^\top g_E,
			\label{eq:exact-spectral-delta-identity}
		\end{align}
		where \(A_E\Sigma_{n,E}A_E=A_E-\lambda_{\mathrm R}A_E^2\).
		
		We first establish sign preservation by a Neumann series argument. By the first order conditions,
		\[
		\hat\delta_E=\frac{\lambda_{\mathrm L}}{\lambda_{\mathrm R}}
		\left(I_{|E|}+\frac{\Sigma_{n,E}}{\lambda_{\mathrm R}}\right)^{-1}s.
		\]
		The assumed restriction on \(\lambda_{\mathrm R}\) implies
		\(r:=\|\Sigma_{n,E}/\lambda_{\mathrm R}\|_\infty<1/2\).
		The Neumann series converges in the induced matrix infinity norm and yields
		\[
		\left(I_{|E|}+\frac{\Sigma_{n,E}}{\lambda_{\mathrm R}}\right)^{-1}s
		=s+\sum_{k=1}^\infty
		\left(-\frac{\Sigma_{n,E}}{\lambda_{\mathrm R}}\right)^ks.
		\]
		Since \(\|s\|_\infty=1\) and the matrix norm is submultiplicative,
		\[
		\left\|\sum_{k=1}^\infty
		\left(-\frac{\Sigma_{n,E}}{\lambda_{\mathrm R}}\right)^ks
		\right\|_\infty
		\leq\sum_{k=1}^\infty r^k=\frac r{1-r}<1.
		\]
		Every coordinate of the vector in parentheses therefore has the same
		strict sign as the corresponding coordinate of \(s\). In particular,
		\begin{equation*}
			\|\hat\delta_E\|_1=s^\top\hat\delta_E
			=\lambda_{\mathrm L}s^\top A_Es.
		\end{equation*}
		Combining~\eqref{eq:exact-spectral-delta-identity} with
		\((A_Es)^\top g_E\leq Z(E)\|A_Es\|_1=Z(E)s^\top A_Es\)
		proves~\eqref{eq:sign-aware-exact-eventwise}. On \(\mathcal E_\tau(E)\)
		we may factor its right side as
		\[
		\lambda_{\mathrm L}^2s^\top A_Es
		\left\{1-2\tau+\lambda_{\mathrm R}
		\frac{s^\top A_E^2s}{s^\top A_Es}\right\},
		\]
		which is strictly positive under~\eqref{eq:sign-aware-positive-criterion}.
		For the simple universal threshold, put \(b=\|\Sigma_{n,E}\|_2\).
		Since \(A_E\) is positive definite and its least eigenvalue is
		\((\lambda_{\mathrm R}+b)^{-1}\),
		\[
		\frac{s^\top A_E^2s}{s^\top A_Es}
		\geq\frac1{\lambda_{\mathrm R}+b}.
		\]
		As \(b\leq\|\Sigma_{n,E}\|_\infty\) and
		\(\lambda_{\mathrm R}>2\|\Sigma_{n,E}\|_\infty\),
		we obtain \(\lambda_{\mathrm R}/(\lambda_{\mathrm R}+b)>2/3\),
		and hence \(\tau_{\mathrm N}(E,s)>5/6\).
		
		It remains to derive the explicit Rayleigh quotient inequality.
		Let \(m=|E|\), put \(q_s=s^\top\Sigma_{n,E}s/m\), and write
		\(\Sigma_{n,E}=U\operatorname{diag}(\mu_1,\ldots,\mu_m)U^\top\),
		with orthonormal eigenvectors \(u_j\). Put
		\(w_j=(u_j^\top s)^2/m\). Then \(\sum_jw_j=1\),
		\(w_j\geq0\), and \(\sum_jw_j\mu_j=q_s\).
		The function
		\(h(x)=(x+2\lambda_{\mathrm R})/(x+\lambda_{\mathrm R})^2\)
		is convex on \([0,\infty)\), because
		\(h''(x)=2(x+4\lambda_{\mathrm R})/(x+\lambda_{\mathrm R})^4>0\).
		Jensen's inequality gives
		\[
		\lambda_{\mathrm L}^2s^\top(A_E+\lambda_{\mathrm R}A_E^2)s
		=\lambda_{\mathrm L}^2m\sum_jw_jh(\mu_j)
		\geq\lambda_{\mathrm L}^2m\,h(q_s).
		\]
		For the stochastic term, let
		\(a=\lambda_{\min}(\Sigma_{n,E})\) and
		\(b=\lambda_{\max}(\Sigma_{n,E})\). If \(a<b\), the convexity of
		\(x\mapsto1/(\lambda_{\mathrm R}+x)\) on \([a,b]\) yields the chord bound
		\begin{align*}
			\frac{s^\top A_Es}{m}
			&=\sum_j\frac{w_j}{\lambda_{\mathrm R}+\mu_j}\\
			&\leq\frac{b-q_s}{b-a}\frac1{\lambda_{\mathrm R}+a}
			+\frac{q_s-a}{b-a}\frac1{\lambda_{\mathrm R}+b}\\
			&=\frac{\lambda_{\mathrm R}+a+b-q_s}
			{(\lambda_{\mathrm R}+a)(\lambda_{\mathrm R}+b)}.
		\end{align*}
		The final expression remains exact when \(a=b\). Inserting these
		estimates in~\eqref{eq:sign-aware-exact-eventwise}
		gives the following additional explicit bound:
		\begin{equation*}
			\begin{aligned}
				\Delta(\hat\beta_{\mathrm L},\hat\beta_{\mathrm R})
				\geq{}&\lambda_{\mathrm L}^2|E|
				\frac{q_s+2\lambda_{\mathrm R}}{(q_s+\lambda_{\mathrm R})^2}\\
				&-2\lambda_{\mathrm L}|E|Z(E)
				\frac{\lambda_{\mathrm R}+\lambda_{\min}(\Sigma_{n,E})
					+\|\Sigma_{n,E}\|_2-q_s}
				{(\lambda_{\mathrm R}+\lambda_{\min}(\Sigma_{n,E}))
					(\lambda_{\mathrm R}+\|\Sigma_{n,E}\|_2)}.
			\end{aligned}
		\end{equation*}
		This form uses the sign direction Rayleigh quotient and the extreme eigenvalues of the selected Gram matrix.
		
	\end{proof}
	
	\subsection*{Proof of Theorem~\ref{thm:lasso-accurate-selection}}
	
	\begin{proof}
		By the exact prediction gap decomposition~\eqref{eq:3} and the deterministic term bound used in the proof of Theorem~\ref{thm:lasso-dmse-lower-bound}, we have:
		\begin{align}
			\mathrm{DMSE}&=\mathbb{E}H_{E,s}(\lambda_{\mathrm{L}},\lambda_{\mathrm{R}})-\frac{2}{n}\mathbb{E}\langle X_E\hat{\delta}_{E},\varepsilon\rangle \notag\\
			&\geq\lambda_{\mathrm{L}}^2\frac{2c+1}{(c+1)^2}\mathbb{P}(E\neq\emptyset)-\frac{2}{n}\mathbb{E}\langle X_E\hat{\delta}_{E},\varepsilon\rangle.
			\label{eq:dmse-after-H-lower-bound}
		\end{align}
		Also, recall that according to equation~\eqref{eq:ridge-correction-closed-form}, we have the closed form expression for the Lasso–Ridge estimator:
		\begin{align*}
			\hat{\delta}_{E}=\lambda_{\mathrm{L}}(\Sigma_{n,E}+ \lambda_{\mathrm{R}}I_{|E|})^{-1}s,
		\end{align*}
		which is determined once the signed pair \((E,s)\) is fixed. Set \(F\coloneqq\{(E,s)=(S_0,s_0)\}\). We consider an alternative approach for the following term, with the tuning parameter \(\lambda_{\mathrm{R}}\) in the second step set to \(c|E|\):
		\begin{align*}
			\frac{1}{n}\mathbb{E}\langle X_E\hat{\delta}_{E},\varepsilon\rangle
			&=\frac{1}{n}\mathbb{E}[\langle X_E\hat{\delta}_{E},\varepsilon\rangle\mathbf{1}_{F}] + \frac{1}{n}\mathbb{E}[\langle X_E\hat{\delta}_{E}, \varepsilon\rangle\mathbf{1}_{F^c}]\\
			&=\frac{1}{n}\mathbb{E}[\langle \lambda_{\mathrm{L}}X_{S_0}(\Sigma_{n,S_0}+ c|S_0|I_{|S_0|})^{-1}s_0,\varepsilon\rangle\mathbf{1}_{F}]\\
			&\quad+\frac{1}{n}\mathbb{E}[\langle  X_E\hat{\delta}_{E}, \varepsilon\rangle\mathbf{1}_{F^c}],
		\end{align*}
		where we adopt the same notation for any index set. In particular, \(\Sigma_{n,S_0}\coloneqq X_{S_0}^\top X_{S_0}/n\), and
		\(s_0\coloneqq\operatorname{sign}(\beta_{0,S_0})\) denotes the fixed reference sign vector, which equals \(s\) on \(F\).
		The vector in the first inner product is deterministic. Since \(\mathbb E\varepsilon=0\), we have
		\begin{align*}
			\frac{1}{n}\mathbb{E}\langle \lambda_{\mathrm{L}}X_{S_0}(\Sigma_{n,S_0}+ c|S_0|I_{|S_0|})^{-1}s_0,\varepsilon\rangle=0,
		\end{align*}
		Substituting this into the previous equation, we obtain
		\begin{align*}
			\frac{1}{n}\mathbb{E}\langle X_E\hat{\delta}_{E},\varepsilon\rangle
			&=\frac{1}{n}\mathbb{E}[\langle X_E\hat{\delta}_{E},\varepsilon\rangle\mathbf{1}_{F^c}]\\
			&\quad-\frac{1}{n}\mathbb{E}[\langle \lambda_{\mathrm{L}}X_{S_0}(\Sigma_{n,S_0}+ c|S_0|I_{|S_0|})^{-1}s_0,\varepsilon\rangle\mathbf{1}_{F^c}].
		\end{align*}
		First, we prove an upper bound of \(\|X_E\hat{\delta}_{E}\|_2^2\). Since \(\|X_E\hat{\delta}_{E}\|_2^2 = 0\) on \(\{E = \emptyset\}\), it is enough to consider the case \(E\neq\emptyset\), where we have the following orthogonal decomposition based on the closed form expression of the Lasso–Ridge estimator in equation~\eqref{eq:ridge-correction-closed-form}:
		\begin{align*}
			\frac{1}{n}\|X_E\hat{\delta}_{E}\|_2^2
			&=\frac{1}{n}\langle X_E\hat{\delta}_{E},X_E\hat{\delta}_{E} \rangle\\
			&=\lambda_{\mathrm{L}}^2s^\top(\Sigma_{n,E}+ \lambda_{\mathrm{R}}I_{|E|})^{-1}\Sigma_{n,E}(\Sigma_{n,E}+ \lambda_{\mathrm{R}}I_{|E|})^{-1}s\\
			&=\lambda_{\mathrm{L}}^2s^\top P \operatorname{diag}\left(\frac{\mu_1}{(\mu_1+\lambda_{\mathrm{R}})^2}, \dots, \frac{\mu_{|E|}}{(\mu_{|E|}+\lambda_{\mathrm{R}})^2}\right)P^\top s,     
		\end{align*}
		where \( \mu_1 \geq \mu_2 \geq \cdots \geq \mu_{|E|}\geq 0 \) denote the ordered eigenvalues of \(\Sigma_{n,E}=X_E^\top X_E/n\) and \(P\) is the corresponding eigenvector matrix. Since the function \(x\mapsto x/(x+a)^2\) attains its maximum at \(x=a\), we have, with \(a=c|E|\),
		\begin{align*}
			\frac{\mu_i}{(\mu_i+c|E|)^2}
			\leq
			\frac{c|E|}{(c|E|+c|E|)^2},
			\qquad i=1,\dots,|E|.
		\end{align*}
		Based on this, and recalling that we choose the tuning parameter \(\lambda_{\mathrm{R}} = c|E|\), we derive
		\begin{align}
			\frac{1}{n}\langle X_E\hat{\delta}_{E},X_E\hat{\delta}_{E} \rangle
			&=\lambda_{\mathrm{L}}^2s^\top P \operatorname{diag}\left(\frac{\mu_1}{(\mu_1+c|E|)^2}, \dots, \frac{\mu_{|E|}}{(\mu_{|E|}+c|E|)^2}\right)P^\top s \notag\\
			&\leq\lambda_{\mathrm{L}}^2s^\top P \operatorname{diag}\left(\frac{c|E|}{(c|E|+c|E|)^2}, \dots, \frac{c|E|}{( c|E|+c|E|)^2}\right)P^\top s \notag\\
			&\leq\lambda_{\mathrm{L}}^2s^\top P \operatorname{diag}\left(\frac{ c|E|}{(2c|E|)^2}, \dots, \frac{ c|E|}{(2c|E|)^2}\right)P^\top s \notag\\
			&\leq\frac{\lambda_{\mathrm{L}}^2}{4c|E|}\|P^\top s\|_2^2 \notag\\
			&\leq\frac{\lambda_{\mathrm{L}}^2}{4c}.
			\label{eq:quadratic-delta-upper-bound}
		\end{align}
		Also note that, since \(\|\varepsilon\|_2^2 / \sigma^2 \sim \chi^2_n\), we have
		\begin{align*}
			\mathbb{E}[\|\varepsilon\|_2^2] = n\sigma^2.
		\end{align*}
		Now returning to our target, we have
		\begin{align}
			\frac{1}{n}\mathbb{E}[\langle  X_E\hat{\delta}_{E},\varepsilon\rangle\mathbf{1}_{F^c}]
			&\leq\frac{1}{n}\sqrt{\mathbb{E}[\langle  X_E\hat{\delta}_{E},\varepsilon\rangle^2]}\sqrt{\mathbb{E}[\mathbf{1}_{F^c}^{2}]} \notag\\
			&=\frac{1}{n}\sqrt{\mathbb{E}[\langle X_E\hat{\delta}_{E},\varepsilon\rangle^2]}\sqrt{\mathbb{P}(F^c)} \notag\\
			&\leq\frac{1}{n}\sqrt{\mathbb{E}[\|X_E\hat{\delta}_{E}\|_2^2\,\|\varepsilon\|_2^2]}\sqrt{\mathbb{P}(F^c)} \notag\\
			&\leq\sqrt{\frac{\lambda_{\mathrm{L}}^2}{4c}\mathbb{E}[\frac{1}{n}\|\varepsilon\|_2^2]}\sqrt{\mathbb{P}(F^c)} \notag \\
			&\leq\frac{\sigma\lambda_{\mathrm{L}}}{2\sqrt{c}}\sqrt{\mathbb{P}(F^c)}.
			\label{eq:false-selection-cross-term-bound}
		\end{align}
		Next, we derive an upper bound for the absolute value of
		\begin{align*}
			\left|\frac{1}{n}\mathbb{E}[\langle \lambda_{\mathrm{L}}X_{S_0}(\Sigma_{n,S_0}+ c|S_0|I_{|S_0|})^{-1}s_0,\varepsilon\rangle\mathbf{1}_{F^c}]\right|
		\end{align*}
		using a similar approach. First, note that on the event \(F\), we have
		\begin{align*}
			\frac{1}{n}\langle X_{E}\hat{\delta}_{E}, X_{E}\hat{\delta}_{E} \rangle
			= \frac{1}{n}\|\lambda_{\mathrm{L}} X_{S_0}(\Sigma_{n,S_0}+ c|S_0|I_{|S_0|})^{-1}s_0\|_2^2 .
		\end{align*}
		Following an argument similar to that used to derive
		\eqref{eq:quadratic-delta-upper-bound}, we obtain
		\begin{align*}
			\frac{1}{n}\|\lambda_{\mathrm{L}}X_{S_0}(\Sigma_{n,S_0}+ c|S_0|I_{|S_0|})^{-1}s_0\|_2^2\leq\frac{\lambda_{\mathrm{L}}^2}{4c}.
		\end{align*}
		Since \(\langle\lambda_{\mathrm{L}}X_{S_0}(\Sigma_{n,S_0}+ c|S_0|I_{|S_0|})^{-1}s_0,\varepsilon\rangle\) is a Gaussian random variable, we have
		\begin{align*}
			\mathbb{E}[\langle\lambda_{\mathrm{L}}X_{S_0}(\Sigma_{n,S_0}+ c|S_0|I_{|S_0|})^{-1}s_0,\varepsilon\rangle^2]
			&=\sigma^2\|\lambda_{\mathrm{L}}X_{S_0}(\Sigma_{n,S_0}+ c|S_0|I_{|S_0|})^{-1}s_0\|_2^2.
		\end{align*}
		Thus for the second term of our final target, we derive
		\begin{align}
			&\left|\frac{1}{n}\mathbb{E}[\langle\lambda_{\mathrm{L}}X_{S_0}(\Sigma_{n,S_0}+ c|S_0|I_{|S_0|})^{-1}s_0,\varepsilon\rangle\mathbf{1}_{F^c}]\right| \notag\\
			&\leq\frac{1}{n}\sqrt{\mathbb{E}[\langle\lambda_{\mathrm{L}}X_{S_0}(\Sigma_{n,S_0}+ c|S_0|I_{|S_0|})^{-1}s_0,\varepsilon\rangle^2]}\sqrt{\mathbb{E}[\mathbf{1}_{F^c}^{2}]} \notag\\
			&\leq\frac{1}{n}\sqrt{\sigma^2\|\lambda_{\mathrm{L}}X_{S_0}(\Sigma_{n,S_0}+ c|S_0|I_{|S_0|})^{-1}s_0\|_2^2}\sqrt{\mathbb{P}(F^c)} \notag\\
			&\leq\frac{\sigma}{\sqrt{n}}\sqrt{\frac{1}{n}\|\lambda_{\mathrm{L}}X_{S_0}(\Sigma_{n,S_0}+ c|S_0|I_{|S_0|})^{-1}s_0\|_2^2}\sqrt{\mathbb{P}(F^c)} \notag\\
			&\leq\frac{\sigma\lambda_{\mathrm{L}}}{2\sqrt{cn}}\sqrt{\mathbb{P}(F^c)}.
			\label{eq:true-support-cross-term-misselection-bound}
		\end{align}
		Combining \eqref{eq:false-selection-cross-term-bound} and
		\eqref{eq:true-support-cross-term-misselection-bound}, we have
		\begin{align*}
			\frac{1}{n}\mathbb{E}\langle  X_E\hat{\delta}_{E},\varepsilon\rangle
			&=\frac{1}{n}\mathbb{E}[\langle  X_E\hat{\delta}_{E},\varepsilon\rangle\mathbf{1}_{F^c}]\\
			&\quad-\frac{1}{n}\mathbb{E}[\langle \lambda_{\mathrm{L}}X_{S_0}(\Sigma_{n,S_0}+ c|S_0|I_{|S_0|})^{-1}s_0,\varepsilon\rangle\mathbf{1}_{F^c}]\\
			&\leq\frac{\sigma\lambda_{\mathrm{L}}}{2\sqrt{c}}\sqrt{\mathbb{P}(F^c)}+\frac{\sigma\lambda_{\mathrm{L}}}{2\sqrt{cn}}\sqrt{\mathbb{P}(F^c)}\\
			&\leq\frac{\sigma\lambda_{\mathrm{L}}}{2\sqrt{c}}\left(1+\frac{1}{\sqrt{n}}\right)\sqrt{\mathbb{P}(F^c)}.
		\end{align*}
		Substituting this into \eqref{eq:dmse-after-H-lower-bound}, we obtain
		\begin{align*}
			\mathrm{DMSE}&=\mathbb{E}H_{E,s}(\lambda_{\mathrm{L}},\lambda_{\mathrm{R}})-\frac{2}{n}\mathbb{E}\langle X_E\hat{\delta}_{E}, \varepsilon\rangle\\
			&\geq\lambda_{\mathrm{L}}^2\frac{2c+1}{(c+1)^2}\mathbb{P}(E\neq \emptyset)-\frac{\sigma\lambda_{\mathrm{L}}}{\sqrt{c}}\left(1+\frac{1}{\sqrt{n}}\right)\sqrt{\mathbb{P}(F^c)}\\
			&=\frac{\lambda_{\mathrm{L}}}{\sqrt{c}}\left(\frac{\sqrt{c}(2c+1)}{(c+1)^2}\lambda_{\mathrm{L}}\mathbb{P}(E\neq \emptyset)-\sigma\left(1+\frac{1}{\sqrt{n}}\right)\sqrt{\mathbb{P}(F^c)}\right).
		\end{align*}
	\end{proof}
	\begin{remark}
		Consider the function
		\begin{align*}
			f(x)=\frac{\sqrt{x}(2x+1)}{(x+1)^2},\qquad x\geq 0.
		\end{align*}
		A direct calculation shows that \(f\) is increasing on 
		\([0,(3+\sqrt{17})/4]\), decreasing thereafter, and attains its maximum 
		value \(f_{\max}\approx 0.79\) at \(x=(3+\sqrt{17})/4\approx 1.78\).
	\end{remark}

	\subsection*{Proof of Proposition~\ref{prop:martingale-localized-max}}
	\begin{proof}
		When \(T=\emptyset\), both sides of the asserted inequality equal zero
		by the zero dimensional convention. Assume henceforth that \(T\neq\emptyset\).
		For \(j\in T\), define
		\[
		\xi_i^j:=\frac{x_{ij}\varepsilon_i}{n},
		\qquad i=1,\dots,n.
		\]
		Since \(X\) is deterministic and \(\mathbb E[\varepsilon_i\mid\mathcal F_{i-1}]=0\), the vectors \(\xi_i=(\xi_i^j)_{j\in T}\) form a \(|T|\)-dimensional martingale difference sequence. Moreover,
		\begin{align*}
			\mathbb E\!\left[
			\max_{j\in T}(\xi_i^j)^2
			\mid\mathcal F_{i-1}
			\right]
			&=
			\frac{\max_{j\in T}x_{ij}^2}{n^2}
			\mathbb E[\varepsilon_i^2\mid\mathcal F_{i-1}] \\
			&\leq
			\frac{K_X^2\sigma^2}{n^2}
			\qquad\text{a.s.}
		\end{align*}
		Hence
		\[
		\mathbb E\sum_{i=1}^n
		\mathbb E\!\left[
		\max_{j\in T}(\xi_i^j)^2
		\mid\mathcal F_{i-1}
		\right]
		\leq
		\frac{K_X^2\sigma^2}{n}.
		\]
		By \citet[Theorem~A1.1.6, p.~219]{nishiyama2021martingale},
		\begin{align*}
			\mathbb E\max_{j\in T}
			\left|\sum_{i=1}^n\xi_i^j\right|
			&\leq
			2\sqrt{2}
			\left\{
			\mathbb E\sum_{i=1}^n
			\mathbb E\!\left[
			\max_{j\in T}(\xi_i^j)^2
			\mid\mathcal F_{i-1}
			\right]
			\log(1+|T|)
			\right\}^{1/2} \\
			&\leq
			2\sqrt{2}\,K_X\sigma
			\sqrt{\frac{\log(1+|T|)}{n}}.
		\end{align*}
		Since
		\[
		\max_{j\in T}\left|\sum_{i=1}^n\xi_i^j\right|
		=
		\left\|\frac{X_T^\top\varepsilon}{n}\right\|_\infty,
		\]
		the claim follows.
	\end{proof}
	
	\subsection*{Proof of Corollary~\ref{cor:martingale-global-dmse}}
	\begin{proof}
		Taking \(T=\{1,\dots,p\}\) in Proposition~\ref{prop:martingale-localized-max} gives
		\[
		\frac{1}{n}\mathbb E\|X^\top\varepsilon\|_\infty
		\leq
		2\sqrt{2}\,K_X\sigma
		\sqrt{\frac{\log(1+p)}{n}}.
		\]
		Substitution into the global expectation lower bound proves the result.
	\end{proof}
	
	\section{Gaussian Maxima and Design Geometry}\label{app:stochastic-term}
	
	\begin{lemma}[Finite sub-Gaussian maximum]\label{lem:subgaussian-infty-bound}
		Let \(Z_1,\dots,Z_m\) be centered random variables satisfying
		\[
		\mathbb E\exp(tZ_j)\leq \exp\!\left(\frac{t^2v^2}{2}\right),
		\qquad t\in\mathbb R,\quad j=1,\dots,m,
		\]
		for some \(v>0\). No independence among the \(Z_j\)'s is required. Then
		\[
		\mathbb E\max_{1\leq j\leq m}|Z_j|
		\leq
		v\sqrt{2\log(2m)}.
		\]
		Consequently, if \(\varepsilon\sim\mathcal N(0,\sigma^2I_n)\) and
		\(X=[X_1,\dots,X_p]\in\mathbb R^{n\times p}\) satisfies
		\(\|X_j\|_2\leq\sqrt n\), then
		\[
		\frac1n\mathbb E\|X^\top\varepsilon\|_\infty
		\leq
		\sigma\sqrt{\frac{2\log(2p)}{n}}.
		\]
		This is the standard finite class sub-Gaussian maximal inequality; see, for example, \citep[Section~17.4, p.~243]{geer2016estimation}.
	\end{lemma}
	
	The following proofs establish the bounds for the concrete design geometries discussed in Section~\ref{subsec:concrete-gaussian-complexity}.
	
	\subsection*{Proof of Proposition~\ref{prop:clustered-gaussian-max}}
	\begin{proof}
		For each column, add and subtract its cluster representative to obtain
		\[
		\frac1nX_j^\top\varepsilon
		=\frac1nv_{r(j)}^\top\varepsilon
		+\frac1n(X_j-v_{r(j)})^\top\varepsilon.
		\]
		By the triangle inequality, taking the maximum and then the expectation gives
		\begin{align*}
			\mathcal G(X)
			&=\mathbb E\max_{j\leq p}\left|\frac1nX_j^\top\varepsilon\right|\\
			&\leq\mathbb E\max_{\ell\leq K}\left|\frac1nv_\ell^\top\varepsilon\right|
			+\mathbb E\max_{j\leq p}\left|\frac1n(X_j-v_{r(j)})^\top\varepsilon\right|.
		\end{align*}
		For each representative, the corresponding Gaussian variable is centered and has variance
		\[
		\operatorname{Var}\!\left(\frac1nv_\ell^\top\varepsilon\right)
		=\frac{\sigma^2\|v_\ell\|_2^2}{n^2}\leq\frac{\sigma^2}{n}.
		\]
		Likewise, the cluster approximation assumption gives
		\[
		\operatorname{Var}\!\left(\frac1n(X_j-v_{r(j)})^\top\varepsilon\right)
		=\frac{\sigma^2\|X_j-v_{r(j)}\|_2^2}{n^2}
		\leq\frac{\sigma^2\delta^2}{n}.
		\]
		Applying Lemma~\ref{lem:subgaussian-infty-bound} separately to the two families, which need not consist of independent variables, yields
		\[
		\mathcal G(X)
		\leq\frac{\sigma}{\sqrt n}\sqrt{2\log(2K)}
		+\frac{\sigma\delta}{\sqrt n}\sqrt{2\log(2p)}.
		\]
		The same argument covers \(\delta=0\), in which case all residual variables vanish.
	\end{proof}
	\subsection*{Proof of Proposition~\ref{prop:approx-low-dimensional-gaussian-max}}
	\begin{proof}
		Write \(\varepsilon=\sigma g\), where \(g\sim\mathcal N(0,I_n)\), and put \(u_j=X_j/\sqrt n\). Column normalization gives \(\|u_j\|_2=1\), and hence
		\[
		\mathcal G(X)=\frac{\sigma}{\sqrt n}
		\mathbb E\max_{j\leq p}|\langle g,u_j\rangle|.
		\]
		For every \(j\), decompose the column into its projection on \(S\) and its orthogonal residual:
		\[
		u_j=P_Su_j+P_S^\perp u_j.
		\]
		The triangle inequality then gives, for every realization of \(g\),
		\begin{align*}
			|\langle g,u_j\rangle|
			&\leq|\langle g,P_Su_j\rangle|
			+|\langle g,P_S^\perp u_j\rangle|\\
			&=|\langle P_Sg,u_j\rangle|
			+|\langle g,P_S^\perp u_j\rangle|\\
			&\leq\|P_Sg\|_2
			+|\langle g,P_S^\perp u_j\rangle|,
		\end{align*}
		where we used the symmetry of \(P_S\) and \(\|u_j\|_2=1\). Since the first term does not depend on \(j\), taking maxima and expectations yields
		\[
		\mathbb E\max_{j\leq p}|\langle g,u_j\rangle|
		\leq\mathbb E\|P_Sg\|_2
		+\mathbb E\max_{j\leq p}|\langle g,P_S^\perp u_j\rangle|.
		\]
		By Cauchy--Schwarz and the fact that \(P_S\) is an orthogonal projection of rank \(r\),
		\[
		\mathbb E\|P_Sg\|_2
		\leq\sqrt{\mathbb E\|P_Sg\|_2^2}
		=\sqrt{\operatorname{tr}(P_S)}=\sqrt r.
		\]
		For the residuals, \(\|P_S^\perp u_j\|_2\leq\delta\) by assumption. Therefore each \(\langle g,P_S^\perp u_j\rangle\) is centered Gaussian with variance at most \(\delta^2\). Lemma~\ref{lem:subgaussian-infty-bound} gives
		\[
		\mathbb E\max_{j\leq p}|\langle g,P_S^\perp u_j\rangle|
		\leq\delta\sqrt{2\log(2p)}.
		\]
		Combining the two bounds and multiplying by \(\sigma/\sqrt n\) proves the claim.
	\end{proof}
	\subsection*{Proof of Proposition~\ref{prop:parametric-gaussian-max}}
	\begin{proof}
		Partition \([0,1]^d\) into \(m^d\) cubes of side length \(1/m\). For each cube containing at least one parameter point \(t_j\), choose its center \(c_\ell\) and the representative \(v_\ell=\sqrt n\,f(c_\ell)\). Because \(f\) takes values on the unit sphere, \(\|v_\ell\|_2=\sqrt n\).
		
		If \(t_j\) lies in the cube centered at \(c_\ell\), then its distance from the center is at most half the cube diagonal:
		\[
		\|t_j-c_\ell\|_2\leq\frac{\sqrt d}{2m}.
		\]
		Using the Lipschitz property of \(f\), we obtain
		\begin{align*}
			\|X_j-v_\ell\|_2
			&=\sqrt n\,\|f(t_j)-f(c_\ell)\|_2\\
			&\leq L\sqrt n\,\|t_j-c_\ell\|_2
			\leq\frac{L\sqrt d}{2m}\sqrt n.
		\end{align*}
		Thus these cubes form at most \(m^d\) clusters whose representatives satisfy the conditions of Proposition~\ref{prop:clustered-gaussian-max}, with approximation radius \(\delta=L\sqrt d/(2m)\). Applying that proposition, and replacing the actual number of nonempty cubes by its upper bound \(m^d\), gives
		\[
		\mathcal G(X)
		\leq\frac{\sigma}{\sqrt n}\left\{
		\sqrt{2\log(2m^d)}
		+\frac{L\sqrt d}{2m}\sqrt{2\log(2p)}\right\}.
		\]
		This is the required inequality.
	\end{proof}
	\subsection*{Proof of Proposition~\ref{prop:union-subspaces-gaussian-max}}
	\begin{proof}
		Write \(\varepsilon=\sigma g\), where \(g\sim\mathcal N(0,I_n)\). For every \(j\), use the subspace \(S_{k(j)}\) specified in the assumption to write
		\[
		u_j=P_{S_{k(j)}}u_j+P_{S_{k(j)}}^\perp u_j.
		\]
		By the triangle inequality and \(\|u_j\|_2=1\),
		\begin{align*}
			|\langle g,u_j\rangle|
			&\leq|\langle P_{S_{k(j)}}g,u_j\rangle|
			+|\langle g,P_{S_{k(j)}}^\perp u_j\rangle|\\
			&\leq\|P_{S_{k(j)}}g\|_2
			+|\langle g,P_{S_{k(j)}}^\perp u_j\rangle|.
		\end{align*}
		Taking maxima and expectations gives
		\[
		\mathbb E\max_j|\langle g,u_j\rangle|
		\leq\mathbb E\max_{k\leq K}\|P_{S_k}g\|_2
		+\mathbb E\max_{j\leq p}|\langle g,P_{S_{k(j)}}^\perp u_j\rangle|.
		\]
		Set \(Y_k=\|P_{S_k}g\|_2\). For each \(k\), Cauchy--Schwarz gives
		\[
		\mathbb EY_k
		\leq\sqrt{\mathbb EY_k^2}
		=\sqrt{\operatorname{tr}(P_{S_k})}\leq\sqrt r.
		\]
		Moreover, \(g\mapsto\|P_{S_k}g\|_2\) is 1-Lipschitz. Gaussian concentration gives, for \(t\in\mathbb R\),
		\[
		\mathbb E\exp\bigl(t(Y_k-\mathbb EY_k)\bigr)\leq\exp(t^2/2).
		\]
		For \(t>0\), the exponential maximum bound therefore implies
		\begin{align*}
			\mathbb E\max_{k\leq K}Y_k
			&\leq\sqrt r
			+\frac1t\log\sum_{k=1}^K
			\mathbb E\exp\bigl(t(Y_k-\mathbb EY_k)\bigr)\\
			&\leq\sqrt r+\frac{\log K}{t}+\frac t2.
		\end{align*}
		For \(K\geq2\), choose \(t=\sqrt{2\log K}\) to obtain \(\mathbb E\max_kY_k\leq\sqrt r+\sqrt{2\log K}\). For \(K=1\), the same bound with \(\sqrt{2\log K}=0\) follows directly from \(\mathbb EY_1\leq\sqrt r\).
		
		Finally, \(\|P_{S_{k(j)}}^\perp u_j\|_2\leq\delta\), so the residual variables are centered Gaussian with variances at most \(\delta^2\). By Lemma~\ref{lem:subgaussian-infty-bound},
		\[
		\mathbb E\max_{j\leq p}|\langle g,P_{S_{k(j)}}^\perp u_j\rangle|
		\leq\delta\sqrt{2\log(2p)}.
		\]
		Combining the two estimates and multiplying by \(\sigma/\sqrt n\) completes the proof.
	\end{proof}
	\subsection*{Proof of Proposition~\ref{prop:path-gaussian-max}}
	\begin{proof}
		Let \(U=\{u_1,\ldots,u_p\}\), with the points ordered as in the assumption. Fix \(\epsilon>0\). Starting from \(u_1\), choose a new representative whenever the accumulated path length from the last representative first exceeds \(\epsilon\). Every point before the next representative is then within Euclidean distance at most \(\epsilon\) of the preceding representative, because Euclidean distance is bounded by the length of the connecting path segment. Each new representative requires more than \(\epsilon\) additional path length. Thus the number of representatives is at most \(1+L_p/\epsilon\), and
		\[
		N(U,\|\cdot\|_2,\epsilon)\leq1+\frac{L_p}{\epsilon}.
		\]
		To handle absolute values, consider the symmetric set \(T=U\cup(-U)\). Covering each copy of \(U\) separately yields
		\[
		N(T,\|\cdot\|_2,\epsilon)
		\leq2\left(1+\frac{L_p}{\epsilon}\right).
		\]
		For \(g\sim\mathcal N(0,I_n)\), we have
		\[
		\max_{j\leq p}|\langle g,u_j\rangle|
		=\sup_{u\in T}\langle g,u\rangle.
		\]
		The canonical metric of this Gaussian process is \(d(u,v)=\|u-v\|_2\), since \(\mathbb E|\langle g,u-v\rangle|^2=\|u-v\|_2^2\). As all points of \(T\) have norm one, its diameter is at most two. Dudley's entropy integral inequality (see \cite{vershynin2018high,talagrand2021upper}) gives, for a universal constant \(C>0\),
		\[
		\mathbb E\max_j|\langle g,u_j\rangle|
		\leq C\int_0^2
		\sqrt{\log N(T,\|\cdot\|_2,\epsilon)}\,d\epsilon.
		\]
		For \(0<\epsilon\leq1\), the covering bound gives
		\[
		\log N(T,\|\cdot\|_2,\epsilon)
		\leq\log\bigl(2(1+L_p)\bigr)+\log(1/\epsilon).
		\]
		For \(1\leq\epsilon\leq2\), monotonicity of covering numbers gives
		\[
		\log N(T,\|\cdot\|_2,\epsilon)
		\leq\log\bigl(2(1+L_p)\bigr).
		\]
		Using \(\sqrt{a+b}\leq\sqrt a+\sqrt b\) and
		\[
		\int_0^1\sqrt{\log(1/\epsilon)}\,d\epsilon=\frac{\sqrt\pi}{2},
		\]
		we conclude, after absorbing numerical factors into \(C\), that
		\[
		\mathbb E\max_j|\langle g,u_j\rangle|
		\leq C\left\{\sqrt{\log\bigl(2(1+L_p)\bigr)}+1\right\}.
		\]
		Since \(\varepsilon=\sigma g\) and \(X_j=\sqrt n\,u_j\), multiplying by \(\sigma/\sqrt n\) completes the proof.
	\end{proof}
	
	\begin{lemma}[Generic chaining bound \cite{talagrand2021upper}]
		Let \((T,d)\) be a finite metric space. Suppose that \(\{X_t:t\in T\}\) is a centered process whose increments are sub-Gaussian with respect to \(d\). Then there exists an absolute constant \(C>0\) such that
		\[
		\mathbb E\sup_{t\in T}X_t
		\leq
		C\,\gamma_2(T,d),
		\]
		where
		\[
		\gamma_2(T,d)
		=
		\inf_{\{T_k\}}
		\sup_{t\in T}
		\sum_{k\geq0}2^{k/2}d(t,T_k),
		\]
		and the infimum is over admissible sequences \(\{T_k\}\), namely \(|T_0|=1\) and \(|T_k|\leq2^{2^k}\) for \(k\geq1\). See also \cite{vershynin2018high} for related entropy and chaining bounds, including Dudley's entropy integral inequality.
	\end{lemma}
	
	\begin{remark}[Gaussian process interpretation]
		Under Gaussian noise, define
		\[
		Z_j:=\frac1nX_j^\top\varepsilon,
		\qquad j=1,\dots,p.
		\]
		Then \(\{Z_j\}_{j=1}^p\) is a centered Gaussian process with canonical metric
		\[
		d(i,j)
		=
		\bigl(\mathbb E|Z_i-Z_j|^2\bigr)^{1/2}
		=
		\frac{\sigma}{n}\|X_i-X_j\|_2.
		\]
		Moreover,
		\[
		\frac1n\|X^\top\varepsilon\|_\infty
		=
		\sup_{j\leq p}|Z_j|.
		\]
		Thus the geometry of the column set induces the increment metric of the stochastic process. To treat the absolute supremum directly, one may equivalently apply the preceding chaining bound to the symmetrized index set corresponding to the columns and their negatives. The finite class inequality at the beginning of this appendix gives the universal cardinality based bound, whereas chaining and entropy methods can exploit the finer metric geometry when many columns are close or lie near a lower complexity family. This is the Gaussian process perspective underlying the structural examples in Section~\ref{subsec:concrete-gaussian-complexity}.
	\end{remark}
	
	\section*{Acknowledgements}
	The author gratefully acknowledges Professor Yoichi Nishiyama for his valuable guidance, insightful comments, and technical assistance.


\begin{thebibliography}{99}
		
		\bibitem[Belloni and Chernozhukov(2013)]{belloni2013least}
		Belloni, A. and Chernozhukov, V. (2013).
		Least squares after model selection in high-dimensional sparse models.
		{\em Bernoulli}, 19(2), 521--547.
		doi:10.3150/11-BEJ410.
		
		\bibitem[Bellec et al.(2018)]{bellec2018slope}
		Bellec, P. C., Lecu{\'e}, G. and Tsybakov, A. B. (2018).
		{SLOPE} meets Lasso: improved oracle bounds and optimality.
		{\em The Annals of Statistics}, 46(6B), 3603--3642.
		
		\bibitem[Cand{\`e}s and Tao(2007)]{candes2007dantzig}
		Cand{\`e}s, E. and Tao, T. (2007).
		The Dantzig selector: statistical estimation when \(p\) is much larger than \(n\).
		{\em The Annals of Statistics}, 35(6), 2313--2351.
		
		\bibitem[Ch{\'e}telat et al.(2017)]{chetelat2017optimal}
		Ch{\'e}telat, D., Lederer, J. and Salmon, J. (2017).
		Optimal two-step prediction in regression.
		{\em Electronic Journal of Statistics}, 11(1), 2519--2546.
		doi:10.1214/17-EJS1287.
		
		\bibitem[Chzhen et al.(2019)]{chzhen2019lasso}
		Chzhen, E., Hebiri, M. and Salmon, J. (2019).
		On Lasso refitting strategies.
		{\em Bernoulli}, 25(4A), 3175--3200.
		doi:10.3150/18-BEJ1085.
		
		\bibitem[Dalalyan et al.(2017)]{dalalyan2017prediction}
		Dalalyan, A. S., Hebiri, M. and Lederer, J. (2017).
		On the prediction performance of the lasso.
		{\em Bernoulli}, 23(1), 552--581.
		
		\bibitem[Hastie et al.(2015)]{hastie2015statistical}
		Hastie, T., Tibshirani, R. and Wainwright, M. (2015).
		{\em Statistical learning with sparsity: The lasso and generalizations}.
		Boca Raton: Chapman and Hall/CRC.
		
		\bibitem[Hebiri and Lederer(2013)]{hebiri2013correlations}
		Hebiri, M. and Lederer, J. (2013).
		How correlations influence lasso prediction.
		{\em IEEE Transactions on Information Theory}, 59(3), 1846--1854.
		
		\bibitem[Liu(2026)]{lassoridge}
		Liu, G. (2026).
		Lasso--Ridge refitting: a two-stage estimator for high-dimensional linear regression.
		{\em Japanese Journal of Statistics and Data Science}, 9, 1197--1219.
		doi:10.1007/s42081-026-00345-1.
		
		\bibitem[Liu and Yu(2013)]{liu2013asymptotic}
		Liu, H. and Yu, B. (2013).
		Asymptotic properties of Lasso+mLS and Lasso+Ridge in sparse high-dimensional linear regression.
		{\em Electronic Journal of Statistics}, 7, 3124--3169.
		doi:10.1214/14-EJS875.
		
		\bibitem[Meinshausen(2007)]{meinshausen2007relaxed}
		Meinshausen, N. (2007).
		Relaxed Lasso.
		{\em Computational Statistics \& Data Analysis}, 52(1), 374--393.
		doi:10.1016/j.csda.2006.12.019.
		
		\bibitem[Nishiyama(2022)]{nishiyama2021martingale}
		Nishiyama, Y. (2022).
		{\em Martingale methods in statistics}.
		Boca Raton: Chapman and Hall/CRC.
		
		\bibitem[Raskutti et al.(2011)]{raskutti2011minimax}
		Raskutti, G., Wainwright, M. J. and Yu, B. (2011).
		Minimax rates of estimation for high-dimensional linear regression over \(\ell_q\)-balls.
		{\em IEEE Transactions on Information Theory}, 57(10), 6976--6994.
		
		\bibitem[Talagrand(2021)]{talagrand2021upper}
		Talagrand, M. (2021).
		{\em Upper and lower bounds for stochastic processes: decomposition theorems}.
		Second edition. Cham: Springer.
		
		\bibitem[Tibshirani(1996)]{tibshirani1996regression}
		Tibshirani, R. (1996).
		Regression shrinkage and selection via the lasso.
		{\em Journal of the Royal Statistical Society: Series B}, 58(1), 267--288.
		
		\bibitem[Tibshirani(2013)]{tibshirani2013lasso}
		Tibshirani, R. J. (2013).
		The lasso problem and uniqueness.
		{\em Electronic Journal of Statistics}, 7, 1456--1490.
		
		\bibitem[van de Geer and Lederer(2013)]{vandegeer2013lasso}
		van de Geer, S. and Lederer, J. (2013).
		The lasso, correlated design, and improved oracle inequalities.
		In {\em From probability to statistics and back: high-dimensional models and processes},
		IMS Collections, 9, 303--316. Institute of Mathematical Statistics.
		
		\bibitem[van de Geer(2016)]{geer2016estimation}
		van de Geer, S. (2016).
		{\em Estimation and testing under sparsity: {\'E}cole d'{{\'E}}t{{\'e}} de probabilit{{\'e}}s de Saint-Flour XLV -- 2015}.
		Cham: Springer.
		
		\bibitem[Vershynin(2018)]{vershynin2018high}
		Vershynin, R. (2018).
		{\em High-dimensional probability: an introduction with applications in data science}.
		Cambridge: Cambridge University Press.
		
		\bibitem[Zhang et al.(2026)]{zhangliu2026descriptive}
		Zhang, Y., Liu, W. and Yang, Y. (2026).
		A unified descriptive-complexity framework for model selection under correlated designs.
		{\em arXiv preprint}, arXiv:2608.26618.
		
		\bibitem[Zhao and Yu(2006)]{zhao2006model}
		Zhao, P. and Yu, B. (2006).
		On model selection consistency of Lasso.
		{\em Journal of Machine Learning Research}, 7, 2541--2563.
		
		\bibitem[Zou(2006)]{zou2006adaptive}
		Zou, H. (2006).
		The adaptive lasso and its oracle properties.
		{\em Journal of the American Statistical Association}, 101(476), 1418--1429.
		
	\end{thebibliography}
\end{document}